\documentclass{birkjour}
\usepackage{graphicx} 
\usepackage{xcolor}
\usepackage{soul}
\usepackage{amsmath}
\usepackage{amsthm}
\usepackage{amssymb}
\usepackage{tikz}
\newtheorem{theorem}{Theorem}[section]
\newtheorem{lemma}[theorem]{Lemma}
\newtheorem{proposition}[theorem]{Proposition}
\newtheorem{corollary}[theorem]{Corollary}
\theoremstyle{definition}
\newtheorem{definition}[theorem]{Definition}
\theoremstyle{remark}
\newtheorem{remark}[theorem]{Remark}
\usepackage[hidelinks]{hyperref}
\numberwithin{equation}{section}

\newcommand{\eps}{\varepsilon}
\newcommand{\dimH}{\dim_{\mathrm{H}}}
\newcommand{\R}{\mathbb{R}}
\newcommand{\N}{\mathbb{N}}
\newcommand{\C}{\mathbb{C}}
\newcommand{\Hol}{H^{1,\mathrm{loc}}}
\newcommand{\Lloc}{L^{2}_{\mathrm{loc}}}
\newcommand{\Lcomp}{L^{2}_{\mathrm{comp}}}

\newcommand{\ri}{\mathrm{i}}
\newcommand{\cS}{\mathcal{S}}
\newcommand{\cA}{\mathcal{A}}
\newcommand{\cH}{\mathcal{H}}
\newcommand{\cF}{\mathcal{F}}

\newcommand{\supp}{\mathrm{supp}}
\newcommand{\spec}{\sigma}
\newcommand{\cond}{\mathrm{cond}}
\newcommand{\rd}{\mathrm{d}}
\newcommand{\re}{\mathrm{e}}
\newcommand{\dist}{\mathrm{dist}}
\newcommand{\rea}{\mathrm{Re}\,}
\newcommand{\bA}{\mathbb{A}}

\newcommand{\bH}{\mathbb{H}}
\newcommand{\bphi}{\boldsymbol{\phi}}
\newcommand{\bpsi}{\boldsymbol{\psi}}

\begin{document}
\title[Wavenumber-explicit estimates in scattering]{Integral equation methods for\\scattering by general compact obstacles:\\ wavenumber-explicit estimates}
\author[S. N. Chandler-Wilde and S. Sadeghi]{Simon N. Chandler-Wilde and Siavash Sadeghi}
\address{Department of Mathematics and Statistics,\\ University of Reading,\\ Whiteknights PO Box 220,\\ Reading RG6 6AX,\\ UK}
\email{s.n.chandler-wilde@reading.ac.uk}
\email{s.sadeghi@pgr.reading.ac.uk}

\begin{abstract}
\small 



There has been significant recent interest in understanding the dependence on the wavenumber, $k$, of boundary integral operators (BIOs), supported on some set $\Gamma\subset \mathbb{R}^n$, that arise in the solution of the Helmholtz equation, $\Delta u + k^2 u=0$. Recently, for the Dirichlet boundary value problem with data $g$, Caetano et al ({\em Proc. R. Soc. A}, 481:20230650, 2025) have proposed a novel integral equation $A_k\phi=g$ that applies for arbitrary compact $\Gamma$. 
In this paper we study the dependence of $A_k$ on $k$, showing that, for $k\geq k_0>0$, $\|A_k\|\leq ck$ while $\|A_k^{-1}\| \leq c'k$ if $\Gamma$ is star-shaped, where $c, c'>0$ depend only on $k_0$ and $\Gamma$. Amongst other bounds we show that: (i) on the one hand, given any mildly increasing unbounded positive sequence $(k_m)$ and any unbounded sequence $(a_m)$, there exists $\Gamma$, with connected complement,  such that $\|A_{k_m}^{-1}\|\geq a_m$ for every $m$; (ii) on the other hand, for every $\Gamma\subset \mathbb{R}^n$ and $k_0,\varepsilon, \delta>0$, there exists $c>0$ and $E\subset [k_0,\infty)$, with Lebesgue measure $m(E)\leq \varepsilon$, such that $\|A_{k}^{-1}\|\leq c k^{2n+2+\delta}$ on $[k_0,\infty)\setminus E$, i.e., the growth of  $\|A_{k}^{-1}\|$ is at worst polynomial in $k$ if one avoids a set $E$ of arbitrarily small measure. 
As a corollary we obtain the first $k$-explicit bounds on the condition number of $S_k$, where $S_k$ is the standard single-layer BIO on $\Gamma$ when $\Gamma$ is  the boundary of a Lipschitz domain, and analogous estimates when $\Gamma$ is a $d$-set (and so of Hausdorff dimension $d$), for non-integer values of $d$.

\end{abstract}

\subjclass{Primary 45B05, 31B10; Secondary 35P25, 78A45, 28A80}

\keywords{boundary integral equation, Helmholtz equation, semiclassical analysis, fractal, scattering}

\date{today}

\maketitle

\tableofcontents

\section{Introduction} \label{sec:intro}

This paper is concerned with the solution, by integral equation (IE) methods, of the exterior sound-soft scattering problem for the Helmholtz equation. In more detail, given some compact obstacle $O\subset \R^n$, where $n\geq 2$ and $\Omega:= \R^n\setminus O$ is connected, and given some wavenumber $k>0$ and some incident field $u^i\in \Hol(\R^n)$ that satisfies the Helmholtz equation
\begin{equation} \label{eq:he0}
\Delta u^i+k^2u^i=0
\end{equation}
in some open neighbourhood
 of $O$, the {\em exterior sound-soft or Dirichlet scattering problem} is the problem of determining the scattered  field $u\in \Hol(\Omega)$ that satisfies the Helmholtz equation
\begin{equation} \label{eq:he}
\Delta u + k^2 u = 0 \quad \mbox{in} \quad \Omega,
\end{equation}
 in a weak or distributional sense,
the Sommerfeld radiation condition
\begin{equation} \label{eq:src}
\frac{\partial u(x)}{\partial r} - \ri k u(x) = o(r^{-(n-1)/2}),  \quad \mbox{uniformly in $\widehat x := x/r$,}\quad \mbox{as} \quad r:= |x|\to \infty,
\end{equation}
 and  the boundary condition  that $u^t:= u^i+u=0$ on $\partial \Omega$, in the sense that $u^t\in \Hol_0(\Omega)$. (We recall the various function space notations that we use in \S\ref{sec:fs} below.)
It is well known that this scattering problem is well--posed;  uniqueness is proved via Green's theorem and Rellich's lemma, and existence and continuous dependence on data is proved, e.g., by {\em a priori} estimates coupled with a limiting absorption argument (see, e.g., \cite{wilcox1975,NR87,mclean2000strongly}). 

Suppose $\Omega$ is a Lipschitz domain, in the usual sense of \cite[Defn.~3.28]{mclean2000strongly}, which requires that $\partial \Omega$  is bounded and is locally, in some neighbourhood of each $x\in \partial \Omega$, the graph of a Lipschitz continuous function in some (possibly rotated) coordinate system, with points in $\Omega$ on just one side of this graph. Then existence of solution can alternatively  be proved by reformulation as a first kind IE involving the   acoustic single-layer boundary integral operator (BIO) (see \eqref{eq:first} and \eqref{eq:Sk} below). Essentially the same IE argument applies when $\Gamma=O$ is an infinitely thin screen with a sufficiently smooth boundary \cite{Stephan1987}. This approach to scattering problems for the Helmholtz equation, via reformulation as a boundary integral equation (BIE), has proved valuable both for the theoretical insights that these formulations provide and as a starting point for practical large-scale computations (e.g., \cite{ColKre,BrKu01,chandler2012numerical,WoGeBeAr2015,ChDeCi2017,AmChLo2019,LiAtDe2024}).

Recently, Caetano et al \cite{caetano2025integral} have shown that these IE formulations, for the case where $\Omega$ is Lipschitz or $O$ is a screen (see also \cite{ClaeysHiptmair2013} for the case where $O$ is a so-called {\em multi-screen}), can be generalised to apply in the case of a general compact obstacle $O\subset \R^n$.
 Let $\cA_k\psi$ denote the acoustic Newtonian potential with density $\psi$, which is defined by 
\begin{equation} \label{eq:Newt}
\mathcal{A}_k\psi(x)=\int_{\mathbb{R}^n}\Phi_k(x,y)\psi(y)\,dy, \quad \mbox{for a.e.} \quad x\in \R^n \quad \mbox{when} \quad \psi\in \Lcomp(\R^n),
\end{equation}
where $\Phi_k(\cdot,\cdot)$ is the (outgoing) fundamental solution of the Helmholtz equation \eqref{eq:he}, given by \eqref{eq:Phidef} below, which reduces  to 
\begin{equation} \label{eq:Phidef2}
\Phi_k(x,y) = \frac{\re^{\ri k|x-y|}}{4\pi|x-y|},  \qquad x,y\in \R^n, \;\; x\neq y,
\end{equation}
in the physically important case $n=3$. 

To obtain the formulation of \cite{caetano2025integral} we choose a compact set $\Gamma$ with
\begin{equation} \label{eq:GamRest}
\partial O \subset \Gamma \subset O.
\end{equation}
(The geometry is illustrated in Figure \ref{fig:only}. Each choice of $\Gamma$ leads to a distinct, valid IE formulation: the choice $\Gamma = \partial O$ leads to an IE equivalent to the single-layer potential IE \eqref{eq:first} in the case that $\Omega$ is Lipschitz; the choice $\Gamma=O$ leads to an IE that is uniquely solvable for all $k>0$.) Having selected $\Gamma$, we
seek a solution to the scattering problem in the form $u|_\Omega$, where $u:=\cA_k\phi$ for some $\phi\in H^{-1}_\Gamma:= \{\psi\in H^{-1}(\R^n):\supp(\psi)\subset \Gamma\}$. As we recall in  \S\ref{sec:IE}, 
this satisfies the scattering problem if $\phi$ satisfies the IE on $\Gamma$,
\begin{equation} \label{eq:iemain}
A_k\phi = g,
\end{equation}
where $A_k$ is a composition of $\cA_k$ with multiplication by a smooth cut-off function $\chi$ and a projection operator $P$ onto a closed subspace $R(P)$ of $H^1(\R^n)$ and $g$ is the projection of $-\chi u^i$ onto that subspace. There is a natural identification of $R(P)$ with the dual space $(H_\Gamma^{-1})^*$ of $H_\Gamma^{-1}$, so that we may view $A_k$ as a continuous mapping from $H^{-1}_\Gamma$ to $(H_\Gamma^{-1})^*$. 

As we will see in \S\ref{sec:IE}, 
 the IE \eqref{eq:iemain} has a solution for all $k>0$, indeed \cite{caetano2025integral} exactly one solution for all $k>0$ if $\Gamma=O$ so that $\Omega_-:= O\setminus \Gamma$ is empty, otherwise exactly one solution if and only if
\begin{equation} \label{eq:SigmaDef}
k^2\not\in \sigma(-\Delta_D(\Omega_-)),
\end{equation}
where $\sigma(-\Delta_D(\Omega_-))$ denotes the spectrum of $-\Delta_D(\Omega_-)$, the negative Dirichlet Laplacian on $L^2(\Omega_-)$.  (Recall (e.g., \cite{levitin2023topics}) that  $\sigma(-\Delta_D(\Omega_-))$ is a countable subset of $(0,\infty)$  whose only accumulation point is $+\infty$, consisting only of eigenvalues, so that $k^2\in \sigma(-\Delta_D(\Omega_-))$ if and only if there exists a non-zero $v\in H_0^1(\Omega_-)$ that satisfies \eqref{eq:he} in $\Omega_-$.) With the convention that $ \sigma(-\Delta_D(\Omega_-))$ $:= \emptyset$ when $\Omega_-=\emptyset$, it follows for all compact $O$ and all compact $\Gamma$ satisfying  \eqref{eq:GamRest}, that, for $k>0$,
\begin{equation} \label{eq:invert}
A_k \mbox{ is invertible if and only if } k^2\not\in \sigma(-\Delta_D(\Omega_-)), \mbox{ where } \Omega_-:= O\setminus \Gamma.
\end{equation}


\begin{figure}
\begin{center}
\begin{tikzpicture}[scale=0.7]
\fill[lightgray] (0,0)--(0,-2)--(2,-1)--(2,1)--(-2,2)--cycle;
\draw[blue,thick] (0,0)--(0,-2)--(2,-1)--(2,1)--(-2,2)--cycle; \draw[blue,thick] (1,0.5)--(2,0.5);
\draw(2,1.4)node[right]{\large \color{blue}$\Gamma$};
\draw(-2.5,-1) node[right]{\large $\Omega$};
\draw(0.5,-0.2) node[right]{\large  $\Omega_-$ };
\end{tikzpicture}
\end{center}
\caption{Geometry schematic. The compact obstacle $O$ is shaded in gray, and $\Omega= O^c:= \R^n\setminus O$ is its connected complement. $\Gamma$ (thick blue lines) is (one possible choice of) the compact set, satisfying \eqref{eq:GamRest},  on which the IE \eqref{eq:iemain} is posed, and $\Omega_-:=O\setminus \Gamma$, so that $O=\Gamma\cup \Omega_-$ and $\Omega\cup \Omega_-=\Gamma^c$.}\label{fig:only}
\end{figure}
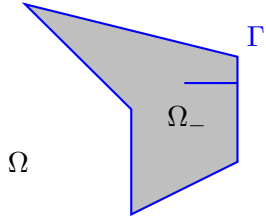


The contribution of the current paper is to obtain the first results on dependence on the wavenumber $k$ for this new IE formulation that applies for any compact obstacle $O$. Indeed, we also obtain the first results on wavenumber dependence for first kind BIEs for acoustic scattering by general Lipschitz obstacles and screens. Our main results are bounds on the norm of  $A_k$, its inverse, and condition number that are explicit in their dependence on $k$, in the large wavenumber regime (i.e., for $k\geq k_0$ for arbitrary fixed $k_0>0$). This leads to corresponding bounds for $S_k$, the classical acoustic single-layer BIO defined by \eqref{eq:Sk} below, and analogous bounds in the case that $\Gamma$ is a $d$-set in the sense of \eqref{eq:dset}, so that $\Gamma$ has fractal dimension $d$, in which case \eqref{eq:iemain} can be written as an IE in which integration is with respect to $d$-dimensional Hausdorff measure. These results all have implications for numerical analysis that we discuss  in  \S\ref{sec:sig}. 

 The investigation of 
the wavenumber dependence of the norms of integral operators and their inverses that appear in  integral equations in acoustics -- the focus of this paper -- dates, for the case of scattering by a ball, to Kress and Spassow \cite{KrSp83,Kr85} and Amini \cite{Am90}.
In the last 20 years there has been significant effort devoted to rigorous analysis of these questions for a variety of boundary conditions and for general classes of scatterers \cite{ChHe:06,BaSa2007,BufSau07,DomGra07,CWMonk2008,ChGrLaLi:09,SpeCha11,BetCha11,chandler2012numerical,Me:12,BoTu2013,spence2014wavenumber,CWHewett2015,HaTa:15,SpKaSm:15,baskin2016sharp,GaSp2019,chandler2020high,lafontaine2021most,GaMaSpe2022,SiavashSimon1}. These investigations have been motivated particularly by numerical analysis questions related to the solution of these BIEs by boundary element methods (BEMs) and other discretisation schemes, and have focussed particularly on the high-wavenumber limit where 3D computation is most challenging, requiring large-scale computational resources (e.g. \cite{WoGeBeAr2015,ChDeCi2017,AmChLo2019,LiAtDe2024}).

Let us outline the rest of the paper. In \S\ref{sec:two} we detail our results on wavenumber dependence, their implications for numerical analysis, and the relationship of our results to the existing literature, detailing the ways in which our results are distinctive, including that they are amongst the first $k$-dependence results for first kind integral equations and the first for integral equations for arbitrary compact scatterers.
In \S\ref{sec:fs} we detail the function space notations and results that we will need to prove our main results in later sections. Section \ref{sec:IE} studies the IE formulation \eqref{eq:iemain} of \cite{caetano2025integral} in more detail, extending the analysis in \cite{caetano2025integral} for $n=2,3$ to higher dimensions, and making other extensions key to our $k$-explicit arguments. Section \ref{sec:main1} provides the proofs of results postponed from \S\ref{sec:main}. In \S\ref{sec:dset} we study the case that $\Gamma$ is a $d$-set (is Ahlfors-David $d$-regular) in the sense of \eqref{eq:dset}, in which case, as shown in \cite{caetano2025integral}, \eqref{eq:iemain} can be written as \eqref{eq:ie22}, where the operator $\bA_k$ that appears is an integral operator with kernel $\Phi_k(x,y)$ and with integration with respect to $d$-dimensional Hausdorff measure. 
In \S\ref{sec:slp2} we establish our bounds on $S_k$ and its inverse, stated in \S\ref{sec:slp}, by application of the bounds in \S\ref{sec:main} and \S\ref{sec:dset}. 

We remark that an announcement of preliminary versions of some of the results of this paper is contained in the conference paper \cite{SiavashSimoncon}.
 
\section{Our main results and their significance} \label{sec:two} 
Below we detail in \S\ref{sec:main} the main results of the paper, namely our bounds on the norm of $A_k$ and its inverse. In \S\ref{sec:sig} we say more about the significance of our results for numerical analysis and make a more detailed comparison with the existing literature, explaining the innovative features of the results of this paper. In \S\ref{sec:slp} we explore the implications of our results for the classical acoustic single-layer BIO $S_k$, defined by \eqref{eq:Sk} below, summarising the bounds on the norm of $S_k$ and its inverse that we obtain as a corollary of our main results, in the cases that $\Omega$ is Lipschitz or the scattering obstacle is a screen, and making comparison to the (few) existing results in this direction. 

The proofs of key bounds on the norm of $A_k$ and its inverse,  namely Propositions \ref{prop:Ak} and \ref{prop:Ak2}, Theorem \ref{thm:main}, and Proposition \ref{prop:lower}, are deferred to \S\ref{sec:main1}. Similarly, the proofs of our bounds on the single-layer BIO $S_k$ and its inverse, Corollaries \ref{thm:Sk0}--\ref{cor:L2}, are deferred to \S\ref{sec:slp2}.

\subsection{Our main results} \label{sec:main}

As outlined in \S\ref{sec:intro}, and discussed in more detail in \S\ref{sec:IE}, $A_k$ is a bounded operator from $H^{-1}_\Gamma$ to $R(P)\subset H^1(\R^n)$, where the closed subspace $R(P)$ is the range of a continuous projection operator $P:H^1(\R^n)\to H^1(\R^n)$ that we will introduce in \S\ref{sec:IE}, where we will identify $R(P)$ as a realisation of $(H_\Gamma^{-1})^*$, the dual space of $H_\Gamma^{-1}$. Further, $A_k^{-1}$ is bounded for $k^2\not\in \sigma(-\Delta_D(\Omega_-))$. The main contribution of this paper is: {\em to obtain bounds on the norm of  $A_k$ and its inverse that are explicit in their dependence on the wavenumber $k$}.  

A complication in this investigation is that there is flexibility in the choice of the projection operator $P$ that appears in the definition \eqref{eq:Akdef2} of $A_k$ below, and that the norm of $A_k$ and its inverse depend on this choice of $P$, as discussed in \S\ref{sec:choice1}. But, once Hilbert space norms have been selected for $H^{\pm 1}(\R^n)$ (for example, the norms \eqref{eq:unorm} or \eqref{eq:knorm}), there is a canonical choice of $P$, namely to choose $P$ to be an orthogonal projection operator, with which choice, as discussed in \S\ref{sec:choice1}: (i) the range $R(P)$ of $P$ is a natural unitary realisation of $(H_\Gamma^{-1})^*$; (ii) the norm of $A_k$ and its inverse coincide with the continuity constant and the reciprocal of the inf-sup constant, respectively, of the sesquilinear form associated to $A_k$ (which sesquilinear form is independent of the choice of $P$). 

Our focus in this paper will be on this canonical choice for $P$ in the case that $H^{\pm 1}(\R^n)$ are equipped with the standard norms \eqref{eq:unorm}, and in the case that they are equipped with the wavenumber-dependent norms \eqref{eq:knorm}. We will denote $P$ and $A_k$ by $P_1$ and $P_k$ and by $A_{1,k}$ and $A_{k,k}$ in these respective cases (see \S\ref{sec:choice1} and \eqref{eq:Akdef}). 

Concretely, as discussed in \S\ref{sec:choice1}, $A_{1,k}$ is a mapping from $H^{-1}_\Gamma$ to  $R(P_1) = \widetilde H^1(\Gamma^c)^\perp$, where 
\begin{equation} \label{eq:Omega*}
\Gamma^c := \R^n\setminus \Gamma, \qquad \widetilde H^1(\Gamma^c):= \text{clos}_{H^1(\R^n)}(C^{\infty}_0(\Gamma^c)),
\end{equation}
and  $\widetilde H^1(\Gamma^c)^\perp$ is the orthogonal complement of  $\widetilde H^1(\Gamma^c)$ in $H^1(\R^n)$ equipped with the norm \eqref{eq:unorm}. Importantly, $\widetilde H^1(\Gamma^c)^\perp$ is a natural realisation in $H^1(\R^n)$, equipped with the standard norm \eqref{eq:unorm}, of the dual space  $(H^{-1}_\Gamma)^*$ (see the discussion below \eqref{eq:dual2}). Throughout the paper when we write $\|A_{1,k}\|$ and $\|A_{1,k}^{-1}\|$ these will denote the standard induced operator norms of $A_{1,k}$ and $A_{1,k}^{-1}$ when $H_\Gamma^{-1}\subset H^{-1}(\R^n)$ and $\widetilde H^1(\Gamma^c)^\perp\subset H^1(\R^n)$ are equipped with the standard norms on $H^{\pm 1}(\R^n)$, given by \eqref{eq:unorm}.

But most of our results, as is usual in the study of $k$-dependence of operator norms (see, e.g., the discussion in \cite{CWHewett2015}), will be for the case when  $H^{\pm 1}(\R^n)$ are equipped with the wavenumber-dependent norms \eqref{eq:knorm}, so that the canonical choice of $P$ is $P=P_k$, with $R(P_k) = \widetilde H^1(\Gamma^c)^{\perp_k}$, the orthogonal complement of $\widetilde H^1(\Gamma^c)$ when $H^1(\R^n)$ is equipped with the norm \eqref{eq:knorm}, and  $A_k=A_{k,k}$. We denote the  induced  operator norms in this case by $\|A_{k,k}\|_k$ and $\|A_{k,k}^{-1}\|_k$, so that
\begin{align} \label{eq:Aknorm}
\|A_{k,k}\|_k  &:= \sup_{0\neq \psi\in H^{-1}_\Gamma}\frac{\|A_{k,k}\psi\|_{H^{1}_k(\R^n)}}{\|\psi\|_{H^{-1}_k(\R^n)}},\\ \label{eq:Aknorm2}
  \|A_{k,k}^{-1}\|_k &:= \sup_{0\neq \psi\in \widetilde H^1(\Omega)^{\perp_k}}\frac{\|A_{k,k}^{-1}\psi\|_{H^{-1}_k(\R^n)}}{\|\psi\|_{H^{1}_k(\R^n)}}.
\end{align}
We emphasise that these norms coincide with the continuity and the reciprocal of the inf-sup constants, respectively, of the sesquilinear form corresponding to $A_{k,k}$; see \eqref{eq:samenorms} and \eqref{eq:samenorms2} applied with $t=k$.

Our main results, summarised below, are stated as bounds on $\|A_{k,k}\|_k$ and $\|A_{k,k}^{-1}\|_k$. But note that these lead immediately to corresponding bounds on $\|A_{1,k}\|$ and $\|A_{1,k}^{-1}\|$; it follows from the inequality \eqref{eq:equiv} below and \eqref{eq:samenorms} and \eqref{eq:samenorms2} applied with $t=1$ and $t=k$, that, for $k>0$, 
\begin{equation} \label{eq:equiv2}
 (\min(1,k^{-1}))^2\|A_{k,k}\|_k\leq \|A_{1,k}\| \leq (\max(1,k^{-1}))^2\|A_{k,k}\|_k
\end{equation}
and
\begin{equation} \label{eq:equiv2a}
(\min(1,k))^2\|A^{-1}_{k,k}\|_k \leq\|A^{-1}_{1,k}\| \leq (\max(1,k))^2\|A^{-1}_{k,k}\|_k.
\end{equation}
Combining bounds on $A_{t,k}$ and $A_{t,k}^{-1}$, for $t=1,k$, we also obtain bounds on 
$$
\cond(A_{1,k}) := \|A_{1,k}\|\, \|A_{1,k}^{-1}\| \quad \mbox{and} \quad \cond_k(A_{k,k}) := \|A_{k,k}\|_k\, \|A_{k,k}^{-1}\|_k,
$$
the condition numbers of  $A_k$ when $P=P_1$ and $P=P_k$, respectively, and when $H^{s}(\R^n)$ is equipped with the norms $\|\cdot\|_{H^{s}(\R^n)}$ and $\|\cdot\|_{H_k^{s}(\R^n)}$, respectively, for $s=\pm 1$. We discuss the value of estimating these condition numbers in \S\ref{sec:sig} below.

The following is our upper bound on the norm of $A_{k,k}$.

\begin{proposition} \label{prop:Ak} 
Given any $k_0>0$, there exists $c>0$ such that
$$
\|A_{k,k}\|_k \leq ck, \qquad k\geq k_0.
$$
\end{proposition}

\noindent In many cases this bound is sharp, as captured in our second result.
\begin{proposition} \label{prop:Ak2} 
Suppose that $\mathrm{int}(\Gamma)$, the interior of $\Gamma$, is non-empty. Then, given any $k_0>0$, there exists $c'>0$ such that
$$
\|A_{k,k}\|_k \geq c'k, \qquad k\geq k_0.
$$
\end{proposition}

Our main bound on $A^{-1}_{k,k}$ is expressed in terms of the norms of (cut-off) resolvents of $-\Delta$. Given a domain $U\subset \R^n$, let $-\Delta_D(U)$ denote the self-adjoint operator that is the negative Dirichlet Laplacian on $L^2(U)$ with domain $\{v\in H_0^1(U):\Delta v \in L^2(U)\}$ (e.g., \cite[\S2.1.6]{levitin2023topics}). Then, for $k^2\not\in \spec(-\Delta_D(U))$, the resolvent $R(k;U):=  (-\Delta_D(U)-k^2)^{-1}:L^2(U)\to L^2(U)$ is well-defined as a bounded operator. Indeed, since $-\Delta_D(U)$ is self-adjoint, 
$$
\|R(k;U)\|_{L^2(U)\to L^2(U)}=\dist\left(k^2,\spec(-\Delta_D(U))\right)^{-1}, \quad k^2\not\in \spec(-\Delta_D(U)).
$$

Our results below make use of $-\Delta_D(U)$ for the case $U=\Omega$ and the case $U=\Omega_-$, with $\Omega_-$  as defined in \eqref{eq:invert}.  (Recall that $\Omega := O^c$ is the complement of the compact obstacle $O\subset \R^n$, and recall our standing assumptions that $\Omega$ is connected and $\Gamma$ satisfies \eqref{eq:GamRest}, so that $\Gamma^c = \Omega \cup \Omega_-$.) Applying the above bound with $U=\Omega_-$ we have, in the case that $\Omega_-$ is non-empty, that
\begin{equation} \label{eq:resol}
\|R(k;\Omega_-)\|_{L^2(\Omega_-)\to L^2(\Omega_-)}=C_k(\Omega_-),
\end{equation}
where 
$$
C_k(\Omega_-) := \left\{\begin{array}{ll}
\left(\dist(k^2,\sigma(-\Delta_D(\Omega_-))\right)^{-1}, & \mbox{if  $\Omega_-$ is non-empty},\\
0, & \mbox{if $\Omega_-=\emptyset$}.\end{array}\right.
$$
In the case that $U=\Omega$ it is well known that $\spec(-\Delta_D(U))=[0,\infty)$, so that $R(k;U)$ is not defined as an operator on $L^2(U)$ for $k$ real. However (e.g., \cite[Lect.~4]{wilcox1975}, \cite{NR87}), for $k>0$, the limit \begin{equation} \label{eq:lap}
v=\lim_{\epsilon\to 0^+}R(k+\ri\epsilon;\Omega)f 
\end{equation}
is well-defined in $\Lloc(\Omega)$ for every $f\in \Lcomp(\Omega)$; indeed, $v$ is the unique solution in $\Hol_0(\Omega)$ to $\Delta v + k^2v = -f$ that satisfies the radiation condition \eqref{eq:src}. Let
$$
 B_R:=\{x\in \R^n:|x|<R\}, \quad \Omega_R:= \Omega \cap B_R,  \quad R_\Gamma := \max_{x\in \Gamma}|x|.
$$
Then, for every $R>R_\Gamma$, there exists a minimal $C_{k,R}(\Omega)>0$ such that, if $\supp(f)\subset \Omega_R$, then 
\begin{equation} \label{eq:CkR}
\|v|_{\Omega_R}\|_{L^2(\Omega_R)}\leq C_{k,R}(\Omega) \|f\|_{L^2(\Omega)},
\end{equation}
so that the resolvent $R(k;\Omega)$ is well-defined for $k>0$ by the limit \eqref{eq:lap} as a continuous mapping $\Lcomp(\Omega)\to \Lloc(\Omega)$. Thus, where $\chi_R$ denotes the characteristic function of $\Omega_R$, the {\em cut-off resolvent} $\chi_R R(k;\Omega)\chi_R$ is well-defined for $k>0$ as a continuous mapping on $L^2(\Omega)$. Indeed, the minimal $C_{k,R}(\Omega)$ in \eqref{eq:CkR} is
\begin{equation} \label{eq:CkR2}
C_{k,R}(\Omega) := \|\chi_R R(k;\Omega)\chi_R\|_{L^2(\Omega)\to L^2(\Omega)}.
\end{equation}

Our main bound on $A_{k,k}^{-1}$ is a bound in terms of these resolvent norms, $C_k(\Omega_-)$ and $C_{k,R}(\Omega)$. We use here and subsequently the notation
\begin{equation} \label{eq:Kkdef}
\Sigma(\Omega_-) := \{k>0: k^2\in \sigma(-\Delta_D(\Omega_-)\}.
\end{equation}
\begin{theorem} \label{thm:main}
Given any $k_0>0$ and $R>R_\Gamma$ there exists $C>0$ such that
$$
\|A_{k,k}^{-1}\|_k \leq Ck^2\left(C_{k,R}(\Omega) + C_k(\Omega_-)\right), \qquad k\in [k_0,\infty)\setminus \Sigma(\Omega_-),
$$
in particular
$$
\|A_{k,k}^{-1}\|_k \leq Ck^2 C_{k,R}(\Omega), \qquad k\in [k_0,\infty),
$$
if $\Gamma=O$ so that $\Omega_-=\emptyset$.
\end{theorem}

The above result is useful because much is known about the growth of $C_{k,R}(\Omega)$ with $k$ and how that depends on the geometry of $O$, in particular whether or not $O$ is trapping and, if $O$ is trapping, the strength of that trapping. For an overview see, e.g., \cite[\S1.1, Table 6.1]{chandler2020high}, \cite[\S1.1]{lafontaine2021most}, or \cite[\S1.2]{SiavashSimon0}; as usual (see \cite{chandler2020high} for more detail) we say that $O$ is nontrapping if, for some $R>R_\Gamma$, all billiard trajectories starting in $\Omega_R$ escape from $\Omega_R$ after some uniform time, and say that $O$ is trapping otherwise. In particular, as shown in \cite{CWMonk2008} (and see \cite[Thm.~1.2]{SiavashSimon0} for the case $n> 3$), for every $R>R_\Gamma$ and $k_0>0$ there exists $C_R>0$ such that
\begin{equation} \label{eq:nt}
C_{k,R}(\Omega) \leq C_Rk^{-1}, \qquad k\geq k_0,
\end{equation}
if the compact set $O$ is star-shaped in the following sense. (The result  in \cite{CWMonk2008} builds on classical results of Morawetz for the case where $O$ is star-shaped and smooth \cite{morawetzludwig1968,morawetz1975}.)
\begin{definition}[Star-shaped] \label{def:ss}
We say that a set $T\subset \R^n$ is {\em star-shaped} if there exists $x\in T$ such that the line segment $[x,y]:= \{tx+(1-t)y:0\leq t\leq 1\}$ is contained in $T$ for every $y\in T$.
\end{definition}
\noindent We note that the bound \eqref{eq:nt} is optimal, in the sense that a simple quasi-mode construction (see, e.g., the discussion before Lemma 3.10 in \cite{CWMonk2008}) shows that, for every compact obstacle $O$, $R>R_\Gamma$, and $k_0>0$, there exists $C_R'>0$ such that 
\begin{equation} \label{eq:nt2}
C_{k,R}(\Omega) \geq C_R'k^{-1}, \qquad k\geq k_0.
\end{equation}

Combining \eqref{eq:nt} with Theorem \ref{thm:main} and Proposition \ref{prop:Ak},  we obtain the following corollary.

\begin{corollary} \label{cor:ss}
If $\Gamma=O$ and $O$ is star-shaped then, given any $k_0>0$, there exists $C>0$ such that
$$
\|A_{k,k}^{-1}\|_k \leq C k, \quad \cond_k(A_{k,k}) \leq C k^2, \qquad k\in [k_0,\infty).
$$
\end{corollary}

A variety of other bounds on $\|A_{k,k}^{-1}\|_k $ can be obtained by combining Theorem \ref{thm:main} with one of the other known bounds for the cut-off resolvent, in particular the bounds summarised in \cite[Table 6.1]{chandler2020high}. All but one of these require that $O$ is $C^\infty$; the two exceptions are the bound of \cite{BaskinWunsch2013} for non-trapping polygons, and the bound for parabolic trapping that is \cite[Thm.~1.1]{chandler2020high}, which requires that $O$ be Lipschitz plus some additional non-trivial technical assumptions. 

We will restrict ourselves to stating two further corollaries. The first of these is obtained by combining Theorem \ref{thm:main} with the well-known worst case (for $C^\infty$ obstacles) cut-off resolvent bound of Burq \cite[Thm.~2]{Burq1998}, that, if $\Omega$ is $C^\infty$ then, for every $k_0>0$ and $R>R_\Gamma$, there exists $C_R>0$ and $\alpha>0$ such that
\begin{equation} \label{eq:worst}
C_{k,R}(\Omega) \leq C_R\re^{\alpha k}, \qquad k\geq  k_0.
\end{equation}
This exponential growth can be achieved as $k\to\infty$ through some sequence in the case of strong trapping (see, e.g., the example for $n=2$ in \cite[Eqn.~(2.20)]{BetCha11}).
Combining Theorem \ref{thm:main} with \eqref{eq:worst} and Proposition \ref{prop:Ak} we obtain the following result.
\begin{corollary} \label{cor:Cinf}
If $\Gamma=O$ and $\Omega$ is a $C^\infty$ domain, then, for every $k_0>0$, there exists $C>0$ and $\alpha>0$ such that
$$
\|A_{k,k}^{-1}\|_k \leq C\re^{\alpha k}, \quad \cond_k(A_{k,k}) \leq C\re^{\alpha k}, \qquad k\geq  k_0.
$$
\end{corollary}

The example for $n=2$ in \cite[Thm.~2.2]{BetCha11} can be tweaked to show that there are (strongly trapping) $C^\infty$ obstacles for which $\|A_k^{-1}\|_k$ grows exponentially as $k\to\infty$ through some sequence, so that Corollary \ref{cor:Cinf} is sharp in some sense. But if smoothness of $O$ is dropped growth can be arbitrarily worse. We prove below the following lower bound. Its proof  adapts an example  for which we prove in \cite{SiavashSimon0} arbitrarily fast increase, as $k\to\infty$ through some sequence, of  $C_{k,R}(\Omega)$.

\begin{proposition} \label{prop:lower}
Suppose that  $(k_m)_{m\in \N}$ and $(a_m)_{m\in \N}$ are increasing, unbounded, positive sequences, and that
\begin{equation} \label{eq:kconstraint}
\sum_{m=1}^\infty k^{-n}_m <\infty.
\end{equation}
Then there exists a compact obstacle $O$ (with $\Omega=O^c$ connected) such that, with $\Gamma=O$,
$$
\|A_{k_m,k_m}^{-1}\|_{k_m} \geq a_m \quad \mbox{and} \quad \cond_{k_m}(A_{k_m,k_m}) \geq a_m, \qquad m\in \N.
$$
\end{proposition}

Our final corollary holds for every compact obstacle $O$ and shows that, in every case, in contrast to Corollary \ref{cor:Cinf} and Proposition \ref{prop:lower}, growth of $\|A_k^{-1}\|_k$ is at worst polynomial in $k$ if a set of wavenumbers of arbitrarily small measure is excluded. Its proof uses the following resolvent estimates proved in our companion paper \cite{SiavashSimon1}. The first estimate  is that, if $\Omega_-$ is non-empty, then, for every $\epsilon>0$ and $k_0>0$, there exists $C>0$ and $E\subset [k_0,\infty)$, with Lebesgue measure $m(E)\leq \epsilon$, such that
\begin{align} \nonumber
\|R(k;\Omega_-)\|_{L^2(\Omega_-)\to L^2(\Omega_-)}&= C_k(\Omega_-)= \left(\dist(k^2,\sigma(-\Delta_D(\Omega_-))\right)^{-1}\\ \label{eq:bounds1}
&\leq C(1+\log^2(k))k^{n-1}, \qquad k\in [k_0,\infty)\setminus E.
\end{align}
Of course the set $E$ in the above inequality must contain a neighbourhood of $\Sigma(\Omega_-)$. This first estimate follows (for details see \cite{SiavashSimon1}) from the Weyl asymptotics for the eigenvalues of $-\Delta_D(\Omega_-)$, recalled in \cite[Thm.~3.3.1, Eqn.~(3.3.4)]{levitin2023topics}.

The second (and rather similar) estimate is that, for every $\epsilon>0$, $\delta>0$, $R>R_\Gamma$, and $k_0>0$, there exists $C>0$ and $E\subset [k_0,\infty)$, with Lebesgue measure $m(E)\leq \epsilon$, such that
\begin{equation} \label{eq:bounds2}
C_{k,R}(\Omega) \leq Ck^{2n+\delta}, \qquad k\in [k_0,\infty)\setminus E.
\end{equation}
This second estimate is a refinement of the main result (Theorem 1.1) of \cite{lafontaine2021most}; our refinement in \cite{SiavashSimon1} is to prove the bound \eqref{eq:bounds2} for arbitrary compact $O$ and to sharpen the exponent (the proof in \cite{lafontaine2021most} establishes \eqref{eq:bounds2} only for Lipschitz $O$, and with our exponent $2n+\delta$ replaced by the larger exponent $5n/2+\delta$).   
Clearly, combining Theorem \ref{thm:main} with \eqref{eq:bounds1} and \eqref{eq:bounds2} and Proposition \ref{prop:Ak} we obtain the following corollary. 
We emphasise that this result holds for every compact $O\subset \R^n$ (with $\Omega=O^c$ connected) and for every compact set $\Gamma$ satisfying \eqref{eq:GamRest}.

\begin{corollary} \label{cor:gen}
For every $\epsilon>0$, $\delta>0$, and $k_0>0$, there exists $C>0$ and $E\subset [k_0,\infty)$, with Lebesgue measure $m(E)\leq \epsilon$, such that
\begin{equation} \label{eq:bounds3}
\|A_{k,k}^{-1}\|_{k}\leq Ck^{2n+2+\delta}, \quad \cond_k(A_{k,k}) \leq Ck^{2n+3+\delta},\qquad k\in [k_0,\infty)\setminus E.
\end{equation}
\end{corollary}

\subsection{Significance and relationship to previous work} \label{sec:sig}

The integral equations that appear in the studies cited in \S\ref{sec:intro}, and our own IE \eqref{eq:iemain}, take the form
\begin{equation} \label{eq:gen}
B_k \phi = g,
\end{equation}
where $B_k:\cH\to \cH^*$  is a continuous mapping from a Hilbert space $\cH$ to its dual space $\cH^*$. In each case $\cH$ is a space of functions defined on $\Gamma$, a compact set satisfying \eqref{eq:GamRest}, with $\Gamma:= \partial \Omega$ in the usual case that the IE is a BIE. In \eqref{eq:gen}, $g\in \cH^*$ is some known data and $\phi$ is the solution to be determined. 

One motivation for estimating $\|B_k\|$ and $\|B_k^{-1}\|$ -- the focus of this paper, and of many of the papers cited in \S\ref{sec:intro} -- is that this leads to estimates for the condition number $\cond(B_k):= \|B_k\|\|B_k^{-1}\|$, and this, by standard estimates for linear operators (e.g., \cite[\S2.7.2]{GVL}) leads to estimates for the perturbation in $\phi$ induced by perturbations in $g$ or the operator $B_k$, this relevant, for  example, to the study of uncertainty quantification (e.g.,  \cite{UQ2,UQ1}). For example, in the simplest case that $g$ is replaced by perturbed data $g^\delta$ in \eqref{eq:gen}, it is elementary that the corresponding perturbed solution $\phi^\delta$ satisfies
$$
\frac{\|\phi-\phi^\delta\|_\cH}{\|\phi\|_\cH} \leq \cond(B_k) \frac{\|g-g^\delta\|_{\cH^*}}{\|g\|_{\cH^*}}.
$$ 

To appreciate the relevance of estimates for $\cond(B_k)$ to the numerical analysis of \eqref{eq:gen}, assume $k>0$ is such that $B_k$ is invertible and consider solution of \eqref{eq:gen} by a least squares method in which one seeks an approximation $\phi_N\in \cH_N$ from some $N$-dimensional subspace of $\cH$, and chooses $\phi_N$ so as to minimise $\|B_k\phi_N-g\|_{\cH^*}$. Equivalently (e.g., \cite[\S3.2.1]{Kirsch}), one chooses $\phi_N \in \cH_N$ so as to solve the variational problem
\begin{equation} \label{eq:ls}
b_k(\phi_N,\psi_N)=(g,B_k\psi_N)_{\cH^*}, \qquad \forall \psi_N\in \cH_N,
\end{equation}
where $(\cdot,\cdot)_{\cH^*}$ is the inner product on $\cH^*$ and $b_k(\cdot,\cdot)$ is the least squares sesquilinear form given by
$$
b_k(\chi,\psi):=(B_k \chi,B_k\psi)_{\cH^*}, \qquad  \chi, \psi\in \cH.
$$
Trivially, $b_k(\cdot,\cdot)$ is symmetric (i.e., $b_k(\chi,\psi)=\overline{b_k(\psi,\chi)}$, for $\chi,\psi\in \cH$), and is continuous and coercive with continuity and coercivity constants $\|B_k\|^2$ and $\|B_k^{-1}\|^{-2}$, i.e.
$$
|b_k(\chi,\psi)| \leq \|B_k\|^2 \|\chi\|_\cH \|\psi\|_{\cH}, \quad |b_k(\psi,\psi)| \geq \|B_k^{-1}\|^{-2}\|\psi\|_{\cH}^2 \qquad \forall \chi, \psi\in \cH.
$$
Thus, by a standard refinement of C\'ea's lemma in the symmetric case (e.g., \cite[(2.8.5), Rem.~2]{BrSc:08}),
$$
\|\phi -\phi_N\|_\cH \leq \cond(B_k) e_N(\phi), \quad \mbox{where} \quad e_N(\phi):= \min_{\psi_N\in \cH_N} \|\phi-\psi_N\|_\cH
$$
is the error in best approximation of $\phi$ from the subspace $\cH_N$. 

To ensure that the error remains small as $k$ increases it is sufficient to increase $N$ sufficiently rapidly (and usually one also chooses $\cH_N$ to depend on $k$) to ensure that $e_N(\phi)$ decreases at least as fast as $(\cond(B_k))^{-1}$. Thus, understanding how $\cond(B_k)$ grows with $k$ is key to the design of a discretisation space $\cH_N$ that will ensure an accurate solution. 

In the 2D case ($n=2$) it is known, for certain geometries of $O$ and choices of $\cH$ (e.g., \cite{HeLaMe:13,ChHeLaTw:15,HeLaCh:15}), how to choose $\cH_N$ dependent on $k$ so that, for every $k_0>0$ and $N_0\in \N$ there exists $c,C,m>0$ such that
$$
e_N(\phi) \leq Ck^m \re^{-cN^{1/2}}, \qquad k\geq k_0, \quad N\geq N_0.
$$
If such an estimate holds, and as long as $\cond(B_k)$ increases at most at a polynomial rate with $k$ (as implied for $\cond_k(A_{k,k})$, outside some set $E$ of small measure, by Proposition \ref{prop:Ak} and Corollary \ref{cor:gen}), it is enough for $N$ to increase in proportion to $\log^2 k$, to ensure, for any given $\epsilon>0$, that $\|\phi -\phi_N\|_\cH\leq \epsilon$ for $k\geq k_0$.   See L\"ohndorf and Melenk \cite{LoMe:11} and the discussions in \cite{chandler2020high,lafontaine2021most} for other, related motivations for establishing a growth of $\cond(B_k)$ with $k$ that is at most polynomial.

Note that understanding conditioning, and its dependence on $k$ and the geometry of $O$, is also important for understanding the convergence of the iterative solvers that are needed to solve the large dense linear systems that typically arise in the numerical solutions of BIE formulations (that all take the form \eqref{eq:gen}) for obstacle scattering problems (see, e.g., \cite{Gr:97,AnDa:21,MaGaSpSp:22}, \cite[Thm.~4.11]{ChSp:24}). (The relevant solvers are the conjugate gradient method and MINRES in the case (as for discretisations of \eqref{eq:ls}) that the matrix is symmetric and positive definite, or GMRES for general complex matrices; see \cite{Gr:97} for specifications of these algorithms.)

The above provides motivation for the results that we have described in \S\ref{sec:main}. Regarding the novelty of our results compared to previous studies of wavenumber dependence of norms and condition numbers of integral operators, and the relationship of this paper to previous work, we make the following remarks.

\begin{remark}[\bf 1st vs.~2nd kind IEs]
This paper is concerned with the $k$-dependence of conditioning of operators arising in first kind IEs. The vast majority of the papers cited in \S\ref{sec:intro} deal with second kind BIE formulations of either the Dirichlet \cite{KrSp83,Kr85,Am90,ChHe:06,BaSa2007,DomGra07,CWMonk2008,ChGrLaLi:09,SpeCha11,BetCha11,chandler2012numerical,Me:12,spence2014wavenumber,HaTa:15,SpKaSm:15,baskin2016sharp,GaSp2019,chandler2020high,lafontaine2021most} or the Neumann \cite{BoTu2013,GaMaSpe2022}  scattering problem. In each of these papers the obstacle $O$ is Lipschitz or smoother, and in each case the integral operator acts on the relatively simple Hilbert space $\cH = L^2(\Gamma)$, where $\Gamma=\partial \Omega$.

Importantly, there are theoretical and computational reasons for preferring first kind IEs, as noted, for example, in \cite{BufSau07}, namely that the associated operators are compact perturbations of coercive operators, as we exhibit, even for general compact obstacles, in Proposition \ref{prop:coer} below.
Thus, by a standard generalisation of C\'ea's lemma (e.g., \cite[Thm.~2.3]{CWSpence22}), all Galerkin methods based on asymptotically dense sequences of subspaces are convergent. By contrast, the operators in the standard second kind IEs are not compact perturbations of coercive operators for general Lipschitz domains, or even for all polyhedra in 3D \cite{CWSpence22}, so that Galerkin methods, in particular based on standard boundary element approximation spaces, need not be convergent \cite[Thm.~1.4]{CWSpence22}.
\end{remark}

\begin{remark}[\bf Previous work on 1st kind IEs]
The only previous $k$-explicit bounds on condition numbers  for first kind IEs are those of Chandler-Wilde and Hewett \cite{CWHewett2015}, building on work of Ha Duong \cite{Ha-Du:90,Ha-Du:92}.
The paper \cite{CWHewett2015} considers first kind BIEs in fractional order Sobolev spaces on $\Gamma=O$ for both the Dirichlet and Neumann problems (for the Dirichlet case the BIO is the operator $S_k$ of \eqref{eq:Sk}),  but for the special case that $O$ is a flat screen, meaning that $O$ is contained in some hyperplane in $\R^n$. In that case the BIOs are convolution operators so that Fourier analysis methods can be used to estimate norms.  (A similar, but only partly rigorous, Fourier analysis is carried out for a modified version of the single-layer BIO $S_k$ in the case that $\Gamma$ is a sphere in \cite{BufSau07}.)
\end{remark}

\begin{remark}[\bf Generality of the geometry in previous results] In contrast to all previous studies of $k$-dependence, the obstacle $O$ in this paper can be any compact subset of $\R^n$. For example, $O$ might be a fractal,  with fractal dimension $d$ taking any value in $[0,n]$, see \S\ref{sec:dset} for more detail (though the case $d<n-2$ is uninteresting as the scattered field is  zero; see  \cite[Rem.~3.5]{caetano2025integral}).  All previous studies of $k$-dependence assume that $\Omega$ is Lipschitz or smoother, with the IE a BIE posed on $\Gamma=\partial \Omega$ which has integer dimension $d=n-1$, or assume \cite{Ha-Du:90,Ha-Du:92,CWHewett2015} that the integral equation is posed on $\Gamma=O$, where $O$ is the closure of some relatively open subset of a hyperplane in $\R^n$ (so that, again, $\Gamma$ has dimension $d=n-1$).
\end{remark}

\begin{remark}[\bf Failure of invertibility for all $k>0$] This is the first paper to derive wavenumber-explicit estimates for the inverse of an integral operator $B_k$ in a case where $B_k$ is not invertible for all $k>0$. This introduces a new element; clearly $\|B_k^{-1}\|$ and $\cond(B_k)$ must blow up at values of $k>0$ where $B_k$ is not invertible.
\end{remark}

\begin{remark}[\bf Results similar to \S\ref{sec:main} for 2nd kind BIEs]
The BIO in the standard 2nd kind BIE on $\Gamma = \partial O$ for the sound-soft scattering problem when $\Omega$ is Lipschitz is $B_k=\frac{1}{2}I+D_k-\ri k S_k$, where $D_k$ is the standard acoustic double-layer BIO. It is shown as \cite[Lemma 6.2]{chandler2020high} (a result that encapsulates arguments in Spence \cite{spence2014wavenumber} for special classes of Lipschitz $\Omega$) that, for every $k_0>0$ there exists $C>0$ such that 
\begin{equation} \label{eq:spence}
\|B_k^{-1}\|_{L^2}\leq Ck^{3/2}C_{k,R}(\Omega), \qquad k\geq k_0,
\end{equation}
where $\|B_k^{-1}\|_{L^2}$ denotes the norm of $B_k^{-1}$ as an operator on $L^2(\Gamma)$. This result, analogous to our Theorem \ref{thm:main}, is proved via bounds on the exterior Dirichlet to Neumann (DtN) map and the interior impedance to Dirichlet map on $\Gamma$, using a representation for $B_k^{-1}$ in terms of these maps \cite[Thm.~2.33]{chandler2012numerical}. In turn, the norm of the DtN map can be bounded in terms of $C_{k,R}(\Omega)$\cite{spence2014wavenumber,chandler2020high}, in part using a $k$-explicit Rellich lemma. None of these tools are available if $\Omega$ is not Lipschitz. But we borrow a key component of the arguments used to prove the DtN bound in \cite[\S3]{spence2014wavenumber} in our proof of Theorem \ref{thm:main}, as detailed at the beginning of \S\ref{sec:Ainv} and in Remark \ref{rem:spence}. 

Just as we have derived corollaries of Theorem \ref{thm:main} in \S\ref{sec:main}, bounds on $\|B_k^{-1}\|_{L_2}$ have been derived by combining \eqref{eq:spence} with known bounds on $C_{k,R}(\Omega)$ for particular geometries of $\Omega$, many of these summarised in \cite[Table 6.1]{chandler2020high}. In particular, a result analogous to Corollary \ref{cor:gen}, proved as \cite[Cor.~2.8]{lafontaine2021most} (with a correction in E.~A.~Spence, private communication), and obtained by combining \eqref{eq:spence} with the bound on $C_{k,R}(\Omega)$ that is \cite[Thm.~1.1]{lafontaine2021most}, is that, for every $\epsilon>0$, $\delta>0$, and $k_0>0$, there exists $C>0$ and $E\subset [k_0,\infty)$, with Lebesgue measure $m(E)\leq \epsilon$, such that 
$$
\|B_k^{-1}\|_{L^2} \leq C k^{5n/2+3/2+\delta}, \qquad k\in [k_0,\infty)\setminus E.
$$
As noted in \cite{SiavashSimon1}, the exponent $5n/2+3/2+\delta$ in the above bound can be reduced to $2n+3/2+\delta$ by combining \eqref{eq:spence} with \eqref{eq:bounds2}, the sharper bound on $C_{k,R}(\Omega)$ shown in \cite{SiavashSimon1}.
\end{remark}


\begin{remark}[\bf Work on numerical solution of \eqref{eq:iemain}] While the motivation for this paper comes in large part from  numerical analysis, we defer further discussion of numerical solution of \eqref{eq:iemain}, by the least squares method of \eqref{eq:ls} or otherwise, to a future paper. But note that numerical solution of  \eqref{eq:iemain} is certainly feasible for complex obstacles, including large classes of fractal obstacles. Indeed, a fully discrete Galerkin method for numerical solution of \eqref{eq:iemain} is described and implemented in \cite{caetano2025integral} for fractals with dimension $d\in (n-2,n]$ that are attractors of iterated function systems of contracting similarities in the sense of \cite[\S2(a)]{caetano2025integral} or \cite[\S9.2]{Falconer2014} (and so are $d$-sets in the sense of \eqref{eq:dset} below). The paper \cite{caetano2025integral} provides a numerical analysis of this Galerkin scheme, but without any consideration of dependence on $k$. The results of this paper should provide the first steps towards a $k$-explicit numerical analysis.
\end{remark}


\subsection{Implications for acoustic single-layer BIOs} \label{sec:slp}

As alluded to in \S\ref{sec:intro}, in the case that $\Omega$ is Lipschitz the scattering problem can  be reformulated as the well-known first kind IE (e.g., \cite[\S3.2.1]{Ned}, \cite[Thm.~9.11]{mclean2000strongly}, \cite[Eqn.~(2.63]{chandler2012numerical})
\begin{equation} \label{eq:first}
S_k\phi=g 
\end{equation}
on $\Gamma=\partial \Omega = \partial O$; here $g:=-u^i|_\Gamma\in H^{1/2}(\Gamma)$ and $S_k:H^{-1/2}(\Gamma)\to H^{1/2}(\Gamma)$ is the acoustic single-layer BIO; explicitly,
\begin{equation} \label{eq:Sk}
S_k\psi(x) := \int_{\Gamma}\Phi_k(x,y)\psi(y)\,ds(y), \qquad x\in \Gamma,
\end{equation}
in the case that $\psi\in L^2(\Gamma)$. (Our boundary Sobolev spaces are as defined in \S\ref{sec:boundary} below.)
The IE \eqref{eq:first} has a solution $\phi\in H^{-1/2}(\Gamma)$ for all $k>0$, and $u=\cS_k \phi\in \Hol(\R^n)$, a single-layer potential with the density $\phi$, is the unique solution to the scattering problem (e.g., \cite[Thm.~9.11]{mclean2000strongly}). Essentially the same IE applies when $\Gamma=O$ is an infinitely thin screen with a sufficiently smooth boundary \cite{Stephan1987}.

Our bounds on the norms of $A_k$ and $A_k^{-1}$ imply wavenumber-explicit bounds on the norms of $S_k$ and $S_k^{-1}$ in the case that $\Gamma$ is the boundary of a Lipschitz domain, is a screen in the sense of \cite{Stephan1987}, or a multi-screen in the sense of \cite{ClaeysHiptmair2013}, via results, in particular the norm-equivalence \eqref{eq:bAknorm}, that apply to a version of the integral operator $S_k$ (denoted $\bA_k$, and see \eqref{eq:bAk2} below) that apply whenever $\Gamma$ is a $d$-set, for some $n-2<d\leq n$, in the sense of \eqref{eq:dset} below; these results are discussed in \S\ref{sec:dset} below. In each of these cases $S_k$ is invertible if and only if $A_k$ is invertible, i.e., if and only if $k\not\in \Sigma(\Omega_-)$. In the Lipschitz case we have the following estimates.

\begin{corollary} \label{thm:Sk0} 
Suppose that $O$ is the closure of a Lipschitz domain and that $\Gamma=\partial O$. Then, for every $k_0>0$ there exists $c,C>0$ such that the norm of $S_k:H^{-1/2}(\Gamma)\to H^{1/2}(\Gamma)$ is bounded by
$$
\|S_k\| \leq c\|A_{k,k}\|_k \leq Ck, \qquad k\geq k_0,
$$
and the norm of its inverse by
$$
\|S_k^{-1}\| \leq Ck^2\|A_{k,k}^{-1}\|_k, \qquad k\in [k_0,\infty)\setminus \Sigma(\Omega_-).
$$
\end{corollary}
\noindent Clearly, estimates for $\|S_k^{-1}\|$ that are explicit in their dependence on $k$ can be obtained by combining the second estimate in the above corollary with the results of \S\ref{sec:main}.

In the case that  $U\subset \R^n$ is a Lipschitz screen  in the sense of \S\ref{sec:boundary} and $\Gamma := O := \overline{U}$ we have the following result, in which $\|S_k\|$ and $\|S_k^{-1}\|$ denote the norms of $S_k:\widetilde H^{-1/2}(U)\to H^{1/2}(U)$ and $S_k^{-1}:H^{1/2}(U)\to \widetilde H^{-1/2}(U)$. In this result we spell out consequences of combining the first bound on $\|S_k^{-1}\|$ with the results of \S\ref{sec:main}.
\begin{corollary} \label{thm:Sk} 
Suppose that $U$ is a Lipschitz screen and $\Gamma=O=\overline{U}$. Then, for every $k_0>0$ there exists $c,C>0$ such that
$$
\|S_k\| \leq c\|A_{k,k}\|_k \leq Ck \quad \mbox{and} \quad \|S_k^{-1}\| \leq Ck^2 \|A_{k,k}^{-1}\|_k, \qquad k\geq k_0.
$$
Further, $S_k$ is invertible for all $k>0$ and, for every $\epsilon>0$, $\delta>0$, and $k_0>0$, there exists $C'>0$ and $E\subset [k_0,\infty)$, with Lebesgue measure $m(E)\leq \epsilon$, such that
\begin{equation} \label{eq:boundsSkgen}
\|S_k^{-1}\|\leq C'k^{2n+4+\delta}, \qquad k\in [k_0,\infty)\setminus E.
\end{equation}
In the case that $U$ and thus $\Gamma$ is star-shaped, there exists, for every $k_0>0$, a $C''>0$ such that 
\begin{equation} \label{eq:Skss}
\|S_k^{-1}\|\leq C''k^3, \qquad k\geq k_0.
\end{equation}
\end{corollary}

We also, as a corollary of Propositions \ref{prop:L2bound} and \ref{prop:L2bound2}, have the following estimate for $\|S_k\|_{L^2}$,  the norm of $S_k:L^2(\Gamma)\to L^2(\Gamma)$. This result is known in dimensions $n=2,3$ \cite[Theorem 3.3]{ChGrLaLi:09}, and follows in that case from Proposition \ref{prop:L2bound2} which gives that $\|S_k\|_{L^2}\leq C k^{(n-3)/2}$, for $k\geq k_0$, for every dimension $n$. If $\Gamma$ is piecewise $C^\infty$ then, for some $C>0$, $\|S_k\|_{L^2} \leq C k^{-1/2}\log(1+k)$, for $k\geq k_0$ \cite{HaTa:15}, a bound that is stronger than the following corollary if $n\geq 3$. Moreover \cite{HaTa:15,GaSp2019}, if $\Gamma$ contains a $C^2$ neighbourhood of a straight line segment, then, for some $C'>0$, $\|S_k\|\geq C' k^{-1/2}$, $k\geq k_0$, so that the estimates below are sharp for $n=2$, and, if they are not sharp for $n\geq 3$ as estimates that hold for all Lipschitz $\Gamma$, fail to be sharp in their $k$-dependence by at worst a factor $k^{1/2+\eps}$ for $n\geq 3$.
\begin{corollary} \label{cor:L2}
Suppose that $O$ is the closure of a  Lipschitz domain and that $\Gamma=\partial O$ or that $U$ is a Lipschitz screen and $\Gamma=O=\overline{U}$. Then, if $n=2,3$, for every $k_0>0$ there exists $c>0$ such that
$$
\|S_k\|_{L^2} \leq c k^{(n-3)/2}, \qquad k\geq k_0.
$$   
If $n\geq 4$ then, for every $\eps>0$ and $k_0>0$ there exists $C>0$ such that
$$
\|S_k\|_{L^2} \leq C k^\eps, \qquad k\geq k_0.
$$
\end{corollary}

We make the following additional remarks regarding the novelty of the above results and their relationship to previous work. 
\begin{remark}[\bf Bounds on $\|S_k\|$] 
Our bound on $S_k$ in Corollary \ref{thm:Sk0} improves, in higher dimensions and for non-smooth boundaries, on the  bounds that have been obtained by Graham et al \cite{GrLoMeSp:15} and by Galkowski and Spence \cite{GaSp2019}. 
The bound of \cite[Thm.~1.6 (i)]{GrLoMeSp:15}, which applies, as does our result, when $\Gamma$ is the boundary of a Lipschitz domain, is that, for some constant $C>0$ dependent on $k_0$, $\|S_k\|\leq Ck^{(n-1)/2}$ for $k\geq k_0$. This is sharper than our bound for dimension $n=2$, and less sharp for $n>3$. The bounds of   \cite[Thm.~1.5]{GaSp2019} are that, for $k\geq k_0$, $\|S_k\|\leq Ck^{1/2}\log(1+k)$ if $\Gamma$ is piecewise smooth, $\|S_k\|\leq Ck^{1/3}\log(1+k)$ if $\Gamma$ is piecewise curved
(as defined in  \cite{GaSp2019}). These bounds are shown in  \cite{GaSp2019} to be sharp, to within the $\log$ factors,  for their respective geometry classes, but the proofs rely on the (piecewise) smoothness of the geometry.

The bound  in Corollary \ref{thm:Sk}  that  $\|S_k\|\leq Ck$, for $k\geq k_0$, can be deduced from the corresponding bound in Corollary \ref{thm:Sk0}. Similarly, bounds on $S_k:\widetilde H^{-1/2}(U)\to H^{1/2}(U)$ can be deduced from the results in \cite{GrLoMeSp:15,GaSp2019}; for example, for $k\geq k_0$, $\|S_k\|\leq Ck^{(n-1)/2}$ and, if $\Gamma$ is piecewise smooth, $\|S_k\|\leq Ck^{1/2}\log(1+k)$. In the special case that $U$ (and so $\Gamma=O$) is contained in a hyperplane, it is shown in \cite{CWHewett2015} that $\|S_k\| \leq Ck^{1/2}$, and that this estimate is sharp.
\end{remark}

\begin{remark}[\bf Bounds on $\|S_k^{-1}\|$]
Our bounds on $\|S_k^{-1}\|$ in Corollary \ref{thm:Sk} are the first $k$-explicit bounds for  a screen that is not flat. The only previous $k$-explicit bound for $\|S_k^{-1}\|$ for a screen is that of \cite{CWHewett2015} for the case that the screen $O$ is contained in a hyperplane. The Fourier analysis in that paper \cite[Thm.~1.7]{CWHewett2015} -- which relies on the flatness of the screen -- implies, together with \eqref{eq:equiv}, that, for every $k_0>0$ there exists $C>0$ such that  $\|S_k^{-1}\|\leq Ck$, $k\geq k_0$. 

Our bounds on $\|S_k^{-1}\|$ in Corollary \ref{thm:Sk0} are the first $k$-explicit bounds for the case that $\Gamma$ is the boundary of a Lipschitz domain. Sharper bounds for this case (sharper by a factor $k^2$) are provided in our companion paper \cite{SiavashSimon1} through additional arguments, including $k$-explicit Rellich identities, that are available in that case  (cf.~\cite[\S3.3]{spence2014wavenumber}, \cite[\S4.2]{chandler2020high}). We emphasise that these methods and the associated sharper bounds do not apply in other cases, in particular do not apply when $\Gamma$ is a Lipschitz screen.
\end{remark}


\section{Function spaces} \label{sec:fs}
We give brief details in this section of the function space notations that we will need to state and prove our results.

\subsection{Sobolev spaces on $\R^n$ and on domains $\Omega\subset \R^n$} \label{sec:fsRn}
Given a domain (i.e., a non-empty open set) $\Omega\subset\mathbb{R}^n$, let $C_0^\infty(\Omega)\subset C^\infty(\R^n)$ denote the set of smooth complex-valued functions which are compactly supported in $\Omega$. 
As usual (e.g., \cite{mclean2000strongly}), let $\cS(\R^n)\supset C_0^\infty(\R^n)$ denote the Schwartz space of rapidly decreasing, $C^\infty$ functions, $\cS^*(\R^n)$ its dual space, the space of tempered distributions. (Throughout, to suit a Hilbert space setting, our distributions and other functionals will be anti-linear rather than linear, so that our dual spaces are spaces of anti-linear continuous functionals.) 

For $u\in \cS(\R^n)$ let $\hat u\in \cS(\R^n)$ denote the Fourier transform of $u$, choosing the normalising constant in our definition so that the mapping $u\mapsto \hat u$ is unitary on $L^2(\R^n)$; precisely, 
$$
\hat u(\xi) := (2\pi)^{-n/2}\int_{\R^n}\re^{-\ri \xi\cdot x}u(x)\, dx, \qquad \xi\in \R^n.
$$
We extend the domain of the Fourier transform to the tempered distributions $\cS^*(\R^n)$ in the usual way. For $s\in \R$, let $H^s(\R^n)$ denote the standard Sobolev space of those $u\in \cS^*(\R^n)$ whose distributional Fourier transform $\hat u$ is locally integrable and satisfies
\begin{equation} \label{eq:unorm}
\|u\|_{H^s(\R^n)} := \left(\int_{\R^n}|\hat u(\xi)|^2(1+|\xi|^2)^s\, \rd \xi\right)^{1/2} < \infty.
\end{equation}
As usual, we identify $L^2(\R^n)$ with $H^0(\R^n)$, so that $H^s(\R^n)\subset L^2(\R^n)$, for $s\geq 0$, and identify $H^{-s}(\R^n)$ with the dual space $(H^s(\R^n))^*$.

We will often equip $H^s(\R^n)$, as is common in the study of $k$-dependence of operators and their norms (see, e.g., the discussion around equation (26) in \cite{CWHewett2015}),  with a $k$-dependent norm given by
\begin{equation} \label{eq:knorm}
\|u\|_{H_k^s(\R^n)} := \left(\int_{\R^n}|\hat u(\xi)|^2(k^2+|\xi|^2)^s\, \rd \xi\right)^{1/2},
\end{equation}
and $H_k^s(\R^n)$ will denote the set $H^s(\R^n)$ equipped with the norm \eqref{eq:knorm}. ($H^s(\R^n)$ and $H^s_k(\R^n)$ are both Hilbert spaces equipped with the inner products implied by their respective norms.)
Clearly, $\|\cdot\|_{H_k^s(\R^n)}= \|\cdot\|_{H^s(\R^n)}$ for $k=1$, and, for every $k>0$, the norms  $\|\cdot\|_{H_k^s(\R^n)}$ and $\|\cdot\|_{H^s(\R^n)}$ are equivalent, with
\begin{equation} \label{eq:equiv}
\min(1,k^{-s})\|u\|_{H_k^s(\R^n)}\leq  \|u\|_{H^s(\R^n)}\leq \max(1,k^{-s})\|u\|_{H_k^s(\R^n)}, \qquad u\in H^s(\R^n).
\end{equation}
For $s\in \R$ and $k>0$, $H_k^{-s}(\R^n)$ is a natural realisation of the dual space of $H_k^s(\R^n)$ through the duality pairing
\begin{equation} \label{eq:dual}
\langle u,v\rangle := \int_{\R^n} \hat u(\xi) \overline{\hat v(\xi)}\, d\xi, \qquad u\in H_k^{-s}(\R^n), \quad v\in H_k^s(\R^n),
\end{equation}
that extends the inner product 
on $L^2(\R^n)$ and the duality pairing on $\cS^*(\R^n)\times \cS(\R^n)$.

We will use at one point that, for every $k>0$, an application of \cite[Cor.~3.2]{chandler2015interpolation} (as corrected in \cite{ChHeMo:22}), gives that $\{H_k^s(\R^n):s\in \R\}$ is an {\em exact interpolation scale} in the sense of \cite[Rem.~3.8]{chandler2015interpolation} (as corrected in \cite{ChHeMo:22}). This means that, for $s,t\in \R$ and $0<\theta<1$, $(H_k^s(\R^n), H_k^t(\R^n))_{[\theta]} = H_k^r(\R^n)$, with equality of norms, where  $r = (1 - \theta)s + \theta t$ and $(H_k^s(\R^n), H_k^t(\R^n))_{[\theta]}$ denotes the result of complex interpolation between the indicated spaces (our notation is that of \cite{Bergh:76}); equivalently, $(H_k^s(\R^n), H_k^t(\R^n))_{[\theta]}$ denotes the space $(H_k^s(\R^n), H_k^t(\R^n))_{\theta,2}$ that is the result of real interpolation with either the $K$- or the $J$-method, with an appropriate normalisation in each case (see \cite{chandler2015interpolation,ChHeMo:22} for details).

Given a domain $\Omega\subset\mathbb{R}^n$, we say that $u\in H^1(\Omega)$ if $u\in L^2(\Omega)$ and its weak derivative $\nabla u\in (L^2(\Omega))^n$. Let $\|\cdot\|_{L^2(\Omega)}$ denote the usual $L^2(\Omega)$ norm. Then $H^1(\Omega)$ is a Hilbert space when equipped with the norm $\|\cdot\|_{H^1(\Omega)}$, defined by
\begin{equation*}
\|u\|^2_{H^1(\Omega)}=\| |\nabla u| \|^2_{L^2(\Omega)}+\| u\|^2_{L^2(\Omega)}, \quad u\in H^1(\Omega).
\end{equation*}
Let
 $$
 H_0^1(\Omega) := \text{clos}_{H^1(\Omega)}(C^{\infty}_0(\Omega)),
 $$
 the closure of $C_0^\infty(\Omega)$ in $H^1(\Omega)$.
 
The $k$-dependent Sobolev space $H_k^1(\Omega)$, $k>0$, is defined as the set $H^1(\Omega)$ equipped with the Hilbert space norm $\|\cdot\|_{H_k^1(\Omega)}$, defined by   
\begin{equation*}
\|u\|^2_{H_k^1(\Omega)}=\| |\nabla u| \|^2_{L^2(\Omega)}+k^2\| u\|^2_{L^2(\Omega)}, \quad u\in H_k^1(\Omega).
\end{equation*}
Note that, in the case $\Omega=\R^n$, these definitions of $H^1(\Omega)$ and $H^1_k(\Omega)$ agree, with equality of norms, with our definitions above for $H^1(\R^n)$ and $H^1_k(\R^n)$.  
Generalising $H_k^1(\Omega)$, we will need, for $m\in \N$ and $k>0$,  at one point the space $H_k^m(\Omega)$, defined, for the case $\Omega\neq \R^n$, as the space of those $u\in L^2(\Omega)$ whose (weak) partial derivatives of all orders $\leq m$ are square integrable on $\Omega$, a Hilbert space equipped with the norm $\|\cdot\|_{H_k^m(\Omega)}$ defined by
$$
\|u\|^2_{H_k^m(\Omega)} := k^m\sum_{0\leq |\alpha|\leq m}\|(k^{-1}\partial)^\alpha u\|^2_{L^2(\Omega)}, \qquad u\in H_k^m(\Omega).
$$
This definition coincides, for $m=1$, with our definition of $H_k^1(\Omega)$, and is consistent (to within equivalence of norms, uniformly for $k>0$) to the definition above of $H^m_k(\R^n)$.
 
 We will use  certain spaces of compactly supported and locally integrable functions. Let $\Lcomp(\Omega)$ denote the space of those functions in $L^2(\Omega)$ that have bounded support, a topological vector space (indeed an LF-space, e.g. \cite{Treves67}) in which a sequence $(u_n)$ is convergent if it is convergent in $L^2(\Omega)$ and, for some bounded $V\subset \Omega$, $\supp(u_n)\subset V$ for each $n$. Let $\Lloc(\Omega)$ denote the set of functions on $\Omega$ that satisfy $u|_V\in L^2(V)$ for every bounded, measurable $V\subset \Omega$, a Fr\'echet space with the obvious topology (e.g., \cite[p.~52]{wilcox1975}) for which a sequence $(u_n)\subset \Lloc(\Omega)$ is convergent if $(u_n|_V)$ is convergent in $L^2(V)$ for every bounded $V\subset \Omega$. Let 
\begin{align*}
 H^{1,\text{loc}}(\Omega)&:= \{u\in \Lloc(\Omega):\nabla u\in (\Lloc(\Omega))^n\}\\ &=\{u\in \Lloc(\R^n): \chi|_\Omega \,u\in H^1(\Omega), \forall \chi\in C^\infty_0(\R^n)\},
\end{align*} 
and let
$$
H_0^{1,\text{loc}}(\Omega):=\{u\in H^{1,\text{loc}}(\Omega): \chi|_\Omega\, u\in H_0^1(\Omega), \forall \chi\in C^\infty_0(\mathbb{R}^n)\}.
$$
  
 \subsection{Subspaces of $H^s(\R^n)$} \label{sec:sub}
We need also certain subspaces of $H^s(\R^n)$. For $s\in\mathbb{R}$ and every domain $\Omega\subset \R^n$, let  
 $$
 \widetilde{H}^s(\Omega):=\text{clos}_{H^s(\R^n)}(C^{\infty}_0(\Omega)).
 $$ 
 Importantly, if $u\in \widetilde{H}^1(\Omega)$, then $u=0$ a.e. on $\mathbb{R}^n\setminus\Omega$ and $u|_\Omega\in H_0^1(\Omega)$. Conversely, if $u\in H_0^1(\Omega)$ and the definition of $u$ is extended to $\R^n$ by setting $u=0$ on $\R^n\setminus \Omega$, then the extended $u\in \widetilde H^1(\Omega)$. 
 
 For any closed set $F\subset\mathbb{R}^n$, we set
\begin{equation*}
H^s_F:=\{u\in H^s(\mathbb{R}^n): \text{supp}(u)\subset F\},
\end{equation*}
noting that $H^s_F$ is trivial, meaning that $H^s_F=\{0\}$, if $\mathrm{int}(F)$, the interior of $F$, is empty and $s$ is large enough (see \cite{HeMo:16}); in particular, $H^{-1}_F=\{0\}$ if and only if $F$ has zero $H^1(\R^n)$ capacity \cite[Thm.~3.15]{HeMo:16}. (Note that  $F$ has positive $H^1(\R^n)$ capacity if $\dim_H(F)>n-2$, but not if $\dim_H(F)<n-2$ \cite[Thm.~2.12]{HeMo:16}, where $\dim_H(F)$ denotes the Hausdorff dimension of $F$, as defined,  e.g., in  \cite{Falconer2014}.) The subspace $\widetilde H^{-s}(\Omega)^\perp \subset H^{-s}(\R^n)$, where $\Omega:=\R^n\setminus F$ and $\perp$ denotes orthogonal complement, is an isometric realisation of the dual space $(H^s_F)^*$ via the duality pairing \eqref{eq:dual} restricted to $\widetilde H^{-s}(\Omega)^\perp \times H^s_F$ \cite[Thm.~3.15]{ChHeMo:17}. 

More generally, arguing as in  \cite[Thm.~3.15]{ChHeMo:17}, and where $\perp_k$ denotes orthogonal complement when $H^s(\R^n)$ is equipped with the norm $H^s_k(\R^n)$, $\widetilde H^{-s}(\Omega)^{\perp_k} \subset H^{-s}(\R^n)$ is an isometric realisation of the dual space $(H^s_F)^*$ via the duality pairing \eqref{eq:dual} restricted to $\widetilde H^{-s}(\Omega)^{\perp_k} \times H^s_F$, when these spaces are equipped with the norms $H^{-s}_k(\R^n)$ and $H^s_k(\R^n)$, respectively. Precisely, for $k>0$, the mapping $\mathcal{I}:\widetilde H^{-s}(\Omega)^{\perp_k} \to (H^s_F)^*$, defined by
$$
\mathcal{I}u(v) := \langle u,v\rangle, \quad v\in H^s_F,
$$
is an isometric isomorphism when $\widetilde H^{-s}(\Omega)^{\perp_k}$ and  $H^s_F$ are equipped with the norms $H^{-s}_k(\R^n)$ and $H^s_k(\R^n)$, enabling us to identify $(H^s_F)^*=(\widetilde H^{-s}(\Omega))^{\perp_k}$, with equality of norms.  

Extending this further, 
let $P:H^{-s}(\R^n)\to H^{-s}(\R^n)$ be any projection operator with closed range $R(P)\subset H^{-s}(\R^n)$ and kernel $\widetilde H^{-s}(\Omega)$, so that the mapping $\widetilde P:\widetilde H^{-s}(\Omega)^{\perp_k}\to R(P)$, defined by $\widetilde Pu :=Pu$, $u\in  \widetilde H^{-s}(\Omega)^{\perp_k}$, is bijective and so an isomorphism by the Banach theorem (e.g., \cite[Corollaries 2.12]{RudinFA}). Then we can identify $R(P)$ with the dual space $(H_F^s)^*$, through the isomorphism
\begin{equation} \label{eq:dual2}
R(P)\to (H_F^s)^*, \; u\mapsto \tilde u, \mbox{ where } \tilde u(v) := (\mathcal{I}\widetilde P^{-1}u)(v)=\langle \widetilde P^{-1}u,v\rangle = \langle P^+u,v\rangle,
\end{equation}
for $u\in R(P)$, $v\in H_F^s$, where $P^+:H^{-s}(\R^n)\to H^{-s}(\R^n)$ is the Moore-Penrose inverse of $P$ (e.g., \cite[\S2.1.2]{Hagen}).

Note that \eqref{eq:dual2}  is an isometric isomorphism in the special case that 
$P$ is orthogonal projection onto $R(P) =  \widetilde H^{-s}(\Omega)^{\perp_k}$, in which case $\widetilde P$ is the identity operator, $P^+=P$, and the isomorphism coincides with $\mathcal{I}$. Conversely, if the isomorphism \eqref{eq:dual2}  is isometric then, since $\mathcal{I}$ is isometric, also $\widetilde P$ is isometric, and so (e.g., \cite[Theorem 12.13]{RudinFA}) unitary, so that, if $u\in \widetilde H^{-s}(\Omega)^{\perp_k}$, then, where $(\cdot,\cdot)_{k}$ denotes the inner product on $H^{-s}(\R^n)$ compatible with the norm $\|\cdot\|_{H_k^{-s}(\R^n)}$, 
$$
\|u-Pu\|^2_{H_k^{-s}(\R^n)} =(u-Pu,u-Pu)_{k}=(P(u-Pu),P(u-Pu))_{k}=0,
$$
since $P$ is a projection. Thus $\widetilde P$ is the identity operator, $R(P)=\widetilde H^{-s}(\Omega)^{\perp_k}$, and $P$ is orthogonal projection onto $\widetilde H^{-s}(\Omega)^{\perp_k}$.



\subsection{Boundary Sobolev spaces} \label{sec:boundary}
We will need in \S\ref{sec:slp2} also Sobolev spaces defined on the boundaries of bounded Lipschitz domains. Suppose that $\Omega^{*}\subset \R^n$ is a bounded Lipschitz domain with boundary $\Gamma = \partial \Omega^{*}$. 
For $|s|\leq1$ we can define boundary Sobolev spaces $H^s(\Gamma)$ on $\Gamma$ in one of several standard equivalent ways (which give the same space, with equivalent norms). Concretely,  we define $H^s(\Gamma)$, for $|s|\leq 1$, as in  \cite[\S A.3]{chandler2012numerical} (or see \cite[pp.~98--99]{mclean2000strongly}), so that $H^s(\Gamma)$ is a Hilbert space for $|s|\leq 1$, $H^s(\Gamma)\subset L^2(\Gamma)$, $0\leq s\leq 1$, indeed $H^0(\Gamma)=L^2(\Gamma)$ with equivalence of norms, $H^s(\Gamma)$ is continuously and densely embedded in $H^t(\Gamma)$, for $-1\leq t< s\leq 1$. and $H^{-s}(\Gamma)=(H^s(\Gamma))^*$, $|s|\leq 1$, with duality pairing $\langle\cdot, \cdot \rangle_\Gamma$ that extends the inner product $(\cdot,\cdot)_{L^2(\Gamma)}$.

We will need also subspaces of $H^s(\Gamma)$, analogous to the subspaces of $H^s(\R^n)$ introduced in \S\ref{sec:sub}. Suppose that $\Gamma_1,\Gamma_2\subset \Gamma$ are disjoint, non-empty, relatively open subsets of $\Gamma$, with $\Pi$ as their common boundary in $\Gamma$, which are such that $(\Gamma_1,\Pi,\Gamma_2)$ is a Lipschitz dissection of $\Gamma$ in the sense of \cite[p.~99]{mclean2000strongly}. (Roughly speaking, this means that $\Gamma_1$ is a subset of $\Gamma$ with a Lipschitz boundary $\Pi$, and we will refer to $\Gamma_1$ as a {\em Lipschitz subset} of $\Gamma$, or as a {\em Lipschitz screen}.)  Let $C_0^\infty(\Gamma):= \{u|_\Gamma:u\in C_0^\infty(\R^n)\}$, $C_0^\infty(\Gamma_1) := \{\phi\in C_0^\infty(\Gamma):\mathrm{supp}(\phi)\subset \Gamma_1\}$ and, for $|s|\leq 1$,
$$
\widetilde H^s(\Gamma_1) :=  \text{clos}_{H^s(\Gamma)}(C^{\infty}_0(\Gamma_1)).
$$
For $|s|\leq 1$, let $H^s(\Gamma_1) := \{\phi|_{\Gamma_1}:\phi\in H^s(\Gamma)\}$, which we equip with the quotient norm $\|\cdot\|_{H^s(\Gamma_1)}$, defined by 
$$
\|\psi\|_{H^s(\Gamma_1)} := \inf_{\stackrel{\phi\in H^s(\Gamma)}{\phi|_{\Gamma_1} = \psi}}\|\phi\|_{H^s(\Gamma)}, \qquad \psi \in H^s(\Gamma_1).
$$ 
$H^s(\Gamma_1)$ is a natural realisation of $(\widetilde H^{-s}(\Gamma_1))^*$, with duality pairing extending the $L^2(\Gamma_1)$ inner product (see \cite[p.~99, Theorem 3.30]{mclean2000strongly}). Our notation $H^1(\Gamma_1)$ is consistent with our notation in \S\ref{sec:fsRn} for $H^1(\Omega)$ by \cite[p.~99, Theorem 3.30]{mclean2000strongly}.

\section{The IE formulation of Caetano et al, and its extension to higher dimensions} \label{sec:IE}
In this section we recall the IE formulation of \cite[\S3(a)]{caetano2025integral} and its justification. As we do this we make non-trivial extensions of the results of \cite{caetano2025integral}. In particular, we provide additional argument that clarifies that the results apply in all space dimensions $n\geq 2$ (the focus in \cite{caetano2025integral} is $n=2,3$). We show that there is large choice available in the selection of the projection operator $P$ in the definition \eqref{eq:Akdef} of the operator $A_k$ below (this flexibility is crucial to our $k$-explicit arguments; see \S\ref{sec:choice1}). We also provide justification that, even when uniqueness fails, the IE has a solution, and that any such solution provides a solution to the scattering problem.

As we recalled in \S\ref{sec:intro}, given some compact obstacle $O\subset \R^n$ and some incident field $u^i\in \Hol(\R^n)$, the approach of Caetano et al \cite{caetano2025integral} to the exterior sound-soft scattering problem is to choose a compact $\Gamma\subset \R^n$ satisfying \eqref{eq:GamRest} and look for a solution in the form $u|_\Omega$, where $u=\cA_k \phi$, for some $\phi\in H^{-1}_\Gamma$. Here $\cA_k$ is the acoustic Newtonian potential, defined by  \eqref{eq:Newt} for $\phi\in \Lcomp(\R^n)$, where $\Phi_k(x,y)$ is the outgoing fundamental solution of the Helmholtz equation in $\R^n$, given by 
(e.g., \cite[Eqn.~(9.14)]{mclean2000strongly})
\begin{equation} \label{eq:Phidef}
\Phi_k(x,y) := \frac{\ri}{4}\left(\frac{k}{2\pi |x-y|}\right)^{(n-2)/2}H_{(n-2)/2}^{(1)}(k|x-y|), \quad x,y\in \R^n, \;\; x\neq y,
\end{equation} 
where $H_m^{(1)}$ denotes the Hankel function of the first kind of order $m$. This reduces (see \cite[Eqn.~10.16.2]{NIST}) to the familiar \eqref{eq:Phidef2} in the physically important case $n=3$.

 It is standard (e.g., \cite[Lem.~3.24]{mclean2000strongly}) that $\Lcomp(\R^n)$ is dense in $H^{s}_{\mathrm{comp}}(\R^n)$ for every $s<0$ and (e.g., \cite[Thm.~3.1.2]{SaSc:11}, \cite[Thm.~6.1]{mclean2000strongly}) that, for every $s\in \R$, $\cA_k$ extends to a continuous map  
\begin{equation} \label{eq:mapping}
\cA_k:H^{s}_{\mathrm{comp}}(\R^n) \to H^{s+2,\mathrm{loc}}(\R^n).
\end{equation}  
Further, for every $\phi\in H^{-1}_\Gamma$, since $\Gamma \subset O$, $\cA_k\phi$ satisfies \eqref{eq:he} (e.g., \cite[Eqns.~(6.2), (9.14)]{mclean2000strongly}); indeed, for every $s\in \R$,
\begin{equation} \label{eq:inv}
(\Delta  + k^2)\cA_k\phi = \cA_k(\Delta  + k^2)\phi = -\phi, \qquad \phi\in H^s_{\mathrm{comp}}(\R^n).
\end{equation}
Moreover, $\cA_k\phi$ satisfies the radiation condition \eqref{eq:src} (e.g.,  \cite[Lem.~7.14, Thm.~9.6]{mclean2000strongly}).  In the proof of Proposition \ref{prop:coer} we will need the following lemma.

\begin{lemma} \label{lem:map}
 $\cA_k:\cS(\R^n)\to C^\infty(\R^n)$, the mapping is continuous, and \eqref{eq:inv} holds for $\phi\in \cS(\R^n)$. 
\end{lemma}
\begin{proof} It follows from \eqref{eq:Phidef} and the definition  and asymptotic properties of the Hankel function \cite[(10.2.2)-(10.2.5)]{NIST} that $\Phi(x,y)=G(x-y)$, with $G\in L^1(\R^n)+L^\infty(\R^n)$, so that \eqref{eq:Newt} can be written as $\cA_k\phi = G*\phi=\phi*G$, where $*$ is convolution (e.g., \cite[p.~58]{mclean2000strongly}). By  Young's inequality (e.g., \cite[Thm.~3.2 and its proof]{mclean2000strongly}), the mapping $\phi\mapsto G*\phi$ is continuous from $\cS(\R^n)\to C(\R)$, providing the unique extension of the definition \eqref{eq:Newt} from $\phi\in C_0^\infty(\R^n)$ to $\phi\in \cS(\R^n)$. Further, by standard Leibniz-rule arguments (cf.~\cite[Thm.~3.3]{mclean2000strongly}), $\cA_k\phi\in C^m(\R^n)$, for every $m\in \N$ and $\phi\in \cS(\R^n)$, with $\partial^\alpha(\cA_k\phi) = (\partial^\alpha \phi)*G$, for every multi-index $\alpha=(\alpha_1,\ldots,\alpha_m)\in \{1,\ldots,n\}^m$, so that, again by Young's inequality, $\cA_k:\cS(\R^n)\to C^\infty(\R^n)$ is continuous. The continuity of this mapping, and that \eqref{eq:inv} holds for the dense subspace $C_0^\infty(\R^n)$ of $\cS(\R^n)$, implies that \eqref{eq:inv} holds also for $\phi\in \cS(\R^n)$.
\end{proof}

It follows from the observations in the paragraph above Lemma \ref{lem:map} that, if $u:=\cA_k \phi$ with $\phi\in H^{-1}_\Gamma$, then $u\in \Hol(\R^n)$, so that $u^t:= u+u^i\in \Hol(\R^n)$, and that $u|_\Omega\in \Hol(\Omega)$ satisfies the exterior sound-soft scattering problem if and only if $u^t|_\Omega\in \Hol_0(\Omega)$.  Further, where $\Gamma^c$ is defined by \eqref{eq:Omega*}, $u^t|_\Omega\in \Hol_0(\Omega)$ if  $u^t|_{\Gamma^c}\in \Hol_0(\Gamma^c)$, since $\Omega\subset \Gamma^c$ and $\partial \Omega = \partial O\subset \Gamma$, so that $\partial \Omega\subset \partial \Gamma^c$.  Let
$$
C^\infty_{0,\Gamma} := \{\chi\in C_0^\infty(\R^n): \chi=1 \mbox{ in a neighbourhood of }\Gamma\},
$$
and note that
\begin{align*}
u^t|_{\Gamma^c}\in \Hol_0(\Gamma^c) &\Leftrightarrow (\chi u^t)|_{\Gamma^c}\in H_0^1(\Gamma^c), \; \forall \chi\in C_0^\infty(\R^n)\\
& \Leftrightarrow  (\chi u^t)|_{\Gamma^c}\in H_0^1(\Gamma^c), \; \mbox{ for some } \chi\in C^\infty_{0,\Gamma},
\end{align*}
since  $(\chi u^t)|_{\Gamma^c}\in H_0^1(\Gamma^c)$, for all $\chi\in C^\infty_{0,\Gamma}$, if and only if  $(\chi u^t)|_{\Gamma^c}\in H_0^1(\Gamma^c)$, for one $\chi\in C^\infty_{0,\Gamma}$.
Thus, fixing  some $\chi\in C^\infty_{0,\Gamma}$, it holds that $u^t|_\Omega\in \Hol_0(\Omega)$ if  $(\chi u^t)|_{\Gamma^c}\in H_0^1(\Gamma^c)$, which holds if $\chi u^t\in \widetilde H^1(\Gamma^c)$.
Let $P:H^1(\R^n)\to H^1(\R^n)$ be some continuous projection operator with kernel $\widetilde H^1(\Gamma^c)$ 
and range $R(P)$ that is a closed subspace of $H^1(\R^n)$ that we may identify with $(H_\Gamma^{-1})^*$ through the isomorphism \eqref{eq:dual2}. Then  $\chi u^t\in \widetilde H^1(\Gamma^c)$ if and only if $P(\chi u^t) = 0$.

 We have shown all but the last sentence of the following result (cf.~\cite[Thm.~3.4, Rem.~3.8]{caetano2025integral}). Note that the definition \eqref{eq:Akdef2} of $A_k$ in this result  is independent of the choice of $\chi$, since $(\chi_1-\chi_2)\Psi \in \widetilde H^1(\Gamma^c)$, so that $P((\chi_1-\chi_2)\Psi)=0$, for all $\Psi\in \Hol(\R^n)$ and $\chi_1,\chi_2 \in C^\infty_{0,\Gamma}$. Thus the following proposition and all subsequent results hold for every  choice of $\chi \in C^\infty_{0,\Gamma}$. We will comment on the choice of $P$ in \S\ref{sec:choice1} below. In \cite{caetano2025integral} $P$ is chosen to be orthogonal projection onto  $R(P)=\widetilde H^1(\Gamma^c)^\perp$, with the orthogonality that of $H^1(\R^n)$ equipped with its standard norm \eqref{eq:unorm}.

\begin{proposition} \label{prop:ie1} 
Suppose that the compact set $\Gamma\subset \R^n$ satisfies \eqref{eq:GamRest} and that $\phi\in H^{-1}_\Gamma$. Then $u|_\Omega$, where $u:=\cA_k \phi$, satisfies the exterior sound-soft scattering problem if 
\begin{equation} \label{eq:iemain2}
A_k\phi = g := -P(\chi u^i),
\end{equation}
where $A_k:H^{-1}_{\Gamma}\to R(P)\subset H^1(\R^n)$ is defined by
\begin{equation} \label{eq:Akdef2}
A_k\psi:=P(\chi \mathcal{A}_k\psi), \qquad \psi \in H^{-1}_\Gamma.
\end{equation}
Conversely, if $u\in \Hol(\Omega)$ satisfies the exterior sound-soft scattering problem and the definition of $u$ is extended to $\R^n$ by setting $u:=-u^i$ on $\R^n\setminus \Omega$, then $u\in \Hol(\R^n)$ and $u=\cA_k\phi$, where $\phi:= -(\Delta+k^2)u\in H^{-1}_{\partial O}\subset H^{-1}_\Gamma$ satisfies \eqref{eq:iemain2}.
\end{proposition}
\begin{proof} It remains to show the last sentence of the proposition. Extend the definition of $u$ to $\R^n$ in the way indicated, and set $u^t:=u+u^i$. Then $(\chi u^t)|_\Omega\in H_0^1(\Omega)$ and $u^t=0$ in $\R^n\setminus \Omega$, so that $\chi u^t\in \widetilde H^1(\Omega)$, so that $u\in \Hol(\R^n)$.  Defining $\phi:= -(\Delta+k^2)u\in H^{-1,\mathrm{loc}}(\R^n)$, we have that $\phi=0$ in $\Omega$ by \eqref{eq:he} and $\phi=-(\Delta+k^2)u^i=0$ by \eqref{eq:he0} in $\R^n\setminus \overline{\Omega}$. Thus $\supp(\phi) \subset \partial O \subset \Gamma$. Further, defining $v:=\cA_k \phi - u\in \Hol(\R^n)$, it follows from \eqref{eq:inv} that $\Delta v + k^2 v=0$ in $\R^n$, and also (see the arguments above Lemma \ref{lem:map}) $v$ satisfies the radiation condition \eqref{eq:src}, so by the uniqueness of solution of  the exterior scattering problem in the simplest case $O=\emptyset$, we have that $v=0$ so that $u=\cA_k\phi$. Thus $A_k \phi = P(\chi u) = P(\chi(u^t-u^i))=-P(\chi u^i)$, since $\chi u^t\in \widetilde H^1(\Omega)\subset \widetilde H^1(\Gamma^c)$. 
\end{proof}

Since the exterior sound-soft scattering problem has a unique solution for all $k>0$, the above proposition has the following corollary (cf.~\cite[Prop.~3.6]{caetano2025integral}). 

\begin{corollary} \label{cor:exists}
Suppose that the compact set $\Gamma\subset \R^n$ satisfies \eqref{eq:GamRest}. Then \eqref{eq:iemain2} has a solution $\phi\in H^{-1}_{\partial O} \subset H^{-1}_\Gamma$, for all $k>0$. 
\end{corollary}

Invertibility of $A_k:H^{-1}_\Gamma\to R(P)$ is studied in \cite{caetano2025integral} via a variational formulation of \eqref{eq:iemain2}. Define, in terms of the duality pairing \eqref{eq:dual}, the sesquilinear form $a_k(\cdot,\cdot)$ on  $H^{-1}_\Gamma \times H^{-1}_\Gamma$ associated to $A_k$ by
\begin{equation} \label{eq:sesqui}
a_k(\varphi,\psi) := \langle A_k\varphi,\psi\rangle, \qquad \varphi,\psi\in H^{-1}_\Gamma,
\end{equation}
for $k>0$. Since $\langle \varphi,\psi\rangle=0$ if $\varphi\in \widetilde H^1(\Gamma^c)$ and $\psi\in H^{-1}_\Gamma$, and $R(Q)=\widetilde H^1(\Gamma^c)$, where $Q:=I-P$, we see that
\begin{equation} \label{eq:ak2}
a_k(\varphi,\psi) = \langle P(\chi \cA_k\varphi),\psi\rangle = \langle \chi \cA_k\varphi,\psi\rangle, \qquad \varphi,\psi\in H^{-1}_\Gamma,
\end{equation}
so that the choice of $P$ does not affect the definition of $a_k(\cdot,\cdot)$.
Further, since $A_k\phi\in R(P)$ and $R(P)$ is a realisation of $(H^{-1}_\Gamma)^*$ through the isomorphism \eqref{eq:dual2}, $\phi\in H^{-1}_\Gamma$ satisfies \eqref{eq:iemain2} if and only if it satisfies the variational equation
\begin{equation} \label{eq:var1}
a_k(\phi,\psi) = \langle g,\psi\rangle, \qquad \forall \psi\in H_\Gamma^{-1},
\end{equation}
which can be written equivalently as
\begin{equation} \label{eq:var2}
\langle \chi \cA_k\varphi,\psi\rangle = -\langle \chi u^i,\psi\rangle, \qquad \forall \psi\in H_\Gamma^{-1}.
\end{equation}
Since $P$ does not appear in \eqref{eq:var2} the solution set of equation \eqref{eq:iemain2}, which is equivalent to \eqref{eq:var2}, is independent of the choice of $P$.

The following result, showing that $A_k$ is a compact perturbation of a continuous, coercive operator, is \cite[Lem.~3.3]{caetano2025integral}. (For detail of our terminology, of compactness and coercivity of operators and associated sesquilinear forms, see, e.g., 
\cite[\S2.2]{CWHeMoBe:21}; in particular we say that an operator $A_k:H_\Gamma^{-1}\to R(P)$ is coercive if the associated sesquilinear form, defined by \eqref{eq:sesqui}, is coercive.) Because the result in  \cite{caetano2025integral}, and the references cited there to justify parts of the proof, are claimed only for $n=2,3$, we provide an independent proof of this proposition (for all $n\geq 2$) modelled in part on \cite[Thm.~6.1]{mclean2000strongly}. 
\begin{proposition} \label{prop:coer}
Suppose that the compact set $\Gamma\subset \R^n$ satisfies \eqref{eq:GamRest}. 
Then, for each $k>0$, the sesquilinear form $a_k(\cdot, \cdot)$ is continuous and compactly perturbed coercive on  $H^{-1}_\Gamma \times H^{-1}_\Gamma$, indeed, for some constants $C_k, \alpha_k > 0$, and some compact sesquilinear form $\tilde a_k(\cdot, \cdot)$,
\begin{align*}
|a_k(\varphi,\psi)| &\leq C_k \|\varphi\|_{H^{-1}(\R^n)}\, \|\psi\|_{H^{-1}(\R^n)},\\  \rea(a_k(\psi,\psi)-\tilde a_k(\psi,\psi)) &\geq \alpha_k \|\psi\|^2_{H^{-1}(\R^n)}, \qquad \forall \varphi,\psi\in H^{-1}_\Gamma.
\end{align*}
\end{proposition}
\begin{proof} The first estimate follows immediately from \eqref{eq:mapping}, or from our proof below of the second estimate. To see the second estimate, we first make a splitting of $\cA_k$ into a part $\cA_k^0:H^{-1}(\R^n)\to H^1(\R^n)$ that is continuous and coercive, and a remainder that is a smoothing operator.
Let $\cF$ denote the Fourier transform operator $u\mapsto \hat u$ on the space $\cS^*(\R^n)$ of tempered distributions; recall that $\cF$ is an isomorphism on both $\cS(\R^n)$ and $\cS^*(\R^n)$ (e.g., \cite{mclean2000strongly}). If $\psi\in \cS(\R^n)$, then
$\cF((\Delta + k^2)\psi)(\xi) = (k^2-|\xi|^2)\hat\psi(\xi)$, for $\xi\in \R^n$. Choose $\sigma\in C_0^\infty(\R^n)$ such that $0\leq \sigma(\xi)\leq 1$, for $\xi\in \R^n$, and such that $\sigma(\xi)=1$, $|\xi|\leq \sqrt{2}k$, $\sigma(\xi)\geq 1/2$, $|\xi|\leq 2k$, $\sigma(\xi)\leq 1/2$, $|\xi|\geq 2k$. Let $\cA_k^0$ be a regularised version of $-(\Delta+k^2)^{-1}$, defined by
$$
\cA_k^0 \psi = \cF^{-1}(\alpha_k^0\cF\psi), \qquad \psi\in C_0^\infty(\R^n), 
$$
where $\alpha_k^0\in C^\infty(\R^n)$ is given by
$$
\alpha_k^0(\xi) := \sigma(\xi) + (1-\sigma(\xi))(|\xi|^2-k^2)^{-1}, \qquad \xi\in \R^n,
$$
so that, for some $\xi_0>2k$,  $\alpha_k^0(\xi) = (|\xi|^2-k^2)^{-1}$, for $|\xi|>\xi_0$.
Then, for some constants $C,\alpha>0$,
$$
C(1+|\xi|^2)^{-1} \geq |a_k^0(\xi)| \geq \rea a_k^0(\xi) \geq \alpha(1+|\xi|^2)^{-1}, \qquad \xi \in \R^n,
$$
so that, for $s\in \R$, $\cA^0_k:H^s(\R^n) \to H^{s+2}(\R^n)$ is continuous. Further, the sesquilinear form on $H^{-1}(\R^n) \times H^{-1}(\R^n)$, defined by $a^0_k(\varphi,\psi) := \langle \cA^0_k\varphi,\psi\rangle$,  $\varphi,\psi\in H^{-1}(\R^n)$, is coercive  since, for $\psi\in C_0^\infty(\R^n)$, which is dense in $H^{-1}(\R^n)$, and where $(\cdot,\cdot)$ denotes the inner product on $L^2(\R^n)$,
$$
\rea a^0_k(\psi,\psi) = \rea(\cF^{-1}(\alpha_k^0\cF\psi),\psi)=\rea(\alpha_k^0\cF\psi,\cF\psi) \geq \alpha \|\psi\|_{H^{-1}(\R^n)}^2.
$$
Define $A_k^0:H^{-1}_\Gamma\to R(P)$ by $A_k^0\psi:= P(\chi \cA_k^0\psi)=P\cA_k^0\psi$, for $\psi\in H^{-1}_\Gamma$, and note that $a_k^0(\cdot,\cdot)$, restricted to $H^{-1}_\Gamma \times H^{-1}_\Gamma$, is coercive, and is the sesquilinear form associated to $A_k^0$, so that (by definition) $A_k^0$ is coercive. 

To complete the proof of the lemma it is enough to show that $A_k-A_k^0$ is compact. Note that multiplication by $\alpha_k^0$ is a continuous operator on $\cS(\R^n)$, so that $\cA_k^0:\cS(\R^n)\to \cS(\R^n)$ and is continuous. Thus, for $\psi\in C_0^\infty(\R^n)$, $\cA_k^0\psi\in \cS(\R^n)$, so that $\cF((\Delta + k^2)\cA^0_k\psi)(\xi) = (k^2-|\xi|^2)\alpha_k^0(\xi)\hat\psi(\xi)$, for $\xi\in \R^n$, 
and, where $\tau\in C_0^\infty(\R^n)$ is defined by $\tau(\xi) := \sigma(\xi)(1+k^2-|\xi|^2)$, $\xi\in \R^n$,
$$
(\Delta+k^2)\cA_k^0\psi = -\psi +K\psi, \quad \mbox{where} \quad   K\varphi:=  \cF^{-1}(\tau\cF\varphi), \qquad \psi \in C_0^\infty(\R^n).
$$
Now, for $s\in \R$,  $\cF:H^s_{\mathrm{comp}}(\R^n)\to C^\infty(\R^n)$ and is continuous; in particular, this holds for $s=-1$ so that $K:H^{-1}_{\mathrm{comp}}(\R^n)\to \cS(\R^n)$ and is continuous; in particular, $K\psi\in \cS(\R^n)$ for $\psi\in C_0^\infty(\R^n)$. Thus, 
multiplying both sides of the above equation by $\cA_k$ and using Lemma \ref{lem:map}, we see that
\begin{equation} \label{eq:Akrep}
\cA_k\psi = \cA_k^0 \psi + \cA_k K\psi, \qquad \psi \in C_0^\infty(\R^n).
\end{equation}
Now $\cA_k$, $\cA_k^0$, and $\cA_k K$ are continuous as mappings $H^{-1}_{\mathrm{comp}}(\R^n)\to \Hol(\R^n)$, indeed $\cA_k K$ is continuous as a mapping $H^{-1}_{\mathrm{comp}}(\R^n)\to C^\infty(\R^n)$, and so as a mapping $H^{-1}_{\mathrm{comp}}(\R^n)\to H^{s,\mathrm{loc}}(\R^n)$, for all $s\in \R$. This implies, since $C_0^\infty(\R^n)$ is dense in $H^{-1}_{\mathrm{comp}}(\R^n)$, that \eqref{eq:Akrep} holds in fact for all $\psi \in H^{-1}_{\mathrm{comp}}(\R^n)$, in particular, for all $\psi \in H^{-1}_{\Gamma}(\R^n)$. Thus $A_k = A_k^0+P\chi \cA_k K$, and  $\chi\cA_k K:H^{-1}_\Gamma\to H^2(\R^n)$ is continuous with range in $H^2_{\supp(\chi)}$, so that, by the compactness of the embedding $H^2_{\supp(\chi)}\subset H^1_{\supp(\chi)}$ \cite[Thm.~3.27(i)]{mclean2000strongly}, $\chi\cA_k K:H^{-1}_\Gamma\to H^1(\R^n)$ is compact, so that $P\chi \cA_k K:H^{-1}_\Gamma\to R(P)$ is compact.
\end{proof}

The following theorem, one part of \cite[Thm.~3.4]{caetano2025integral}, is, in part, a corollary of the above result, which implies that $A_k=A_k^0+\widetilde A_k$, where $\widetilde A_k$ is compact and $A_k^0$ is continuous and coercive. This in turn implies, by the Lax-Milgram lemma  (e.g., \cite[Lem.~2.32]{mclean2000strongly}), that $A_k^0$ is invertible, so that  (e.g., \cite[Thm.~2.33]{mclean2000strongly}) $A_k$ is Fredholm of index zero. The proof in \cite[Thm.~3.4]{caetano2025integral} that $A_k$ is injective (and so invertible) if and only if  $k^2\not\in \sigma(-\Delta_D(\Omega_-))$ is claimed only for $n=2,3$, but is valid for all $n\geq 2$. Recall from \S\ref{sec:intro} that 
\begin{equation} \label{eq:Omegam}
\Omega_-:= O\setminus \Gamma,
\end{equation} 
and that $\sigma(-\Delta_D(\Omega_-)):= \emptyset$ if $\Omega_-$ is empty.

\begin{theorem} \label{thm:invert}
Suppose that the compact set $\Gamma\subset \R^n$ satisfies \eqref{eq:GamRest}. Then $A_k:H^{-1}_\Gamma\to  R(P)$ is Fredholm of index zero for $k>0$. Further, $A_k$ is injective, and so invertible, if and only if $k^2\not\in \sigma(-\Delta_D(\Omega_-))$. 
\end{theorem}

Of course this result implies that $A_k$ is invertible for all $k>0$ in the case that $\Gamma=O$. Combined with Corollary \ref{cor:exists} and Proposition \ref{prop:ie1}, Theorem \ref{thm:invert} has the following straightforward corollary.

\begin{corollary} \label{cor:IEfinal}
Suppose that the compact set $\Gamma\subset \R^n$ satisfies \eqref{eq:GamRest}. Then \eqref{eq:iemain2} has a solution $\phi\in H^{-1}_{\partial O} \subset H^{-1}_\Gamma$, for all $k>0$, indeed this is the unique solution $\phi\in H^{-1}_\Gamma$ of \eqref{eq:iemain2} if $k^2\not\in \sigma(-\Delta_D(\Omega_-))$, while \eqref{eq:iemain2} has infinitely many solutions  $\phi\in H^{-1}_\Gamma$ if $k^2\in \sigma(-\Delta_D(\Omega_-))$. For all $k>0$, if  $\phi\in H^{-1}_\Gamma$ satisfies \eqref{eq:iemain2} and $u=\cA_k\phi$, then $u|_\Omega$ is the unique solution of the exterior sound-soft scattering problem.
\end{corollary}


\subsection{The choice of $P$} \label{sec:choice1}
There are many choices for the projection operator in the above formulation and the definition \eqref{eq:Akdef2} of $A_k$. For each  choice of $P$, the range $R(P)$ is a realisation of $(H_\Gamma^{-1})^*$, through the isomorphism \eqref{eq:dual2}.  While each choice of $P$ leads to a different definition of $A_k$, a consequence of Theorem \ref{thm:invert} is that, if $A_k$ is invertible for one choice of $P$ it is invertible for every  $P$. Further, as already noted, it follows from the equivalence of \eqref{eq:iemain2} with the variational formulation \eqref{eq:var2}, 
that the solution set of \eqref{eq:iemain2} is independent of $P$. 

But the choice of $P$ is important in the context of this paper because it influences the norm of $A_k$ and its inverse. As discussed in \S\ref{sec:sub}, there is precisely  one choice of $P$ for which $R(P)$ is a unitary realisation of $(H^{-1}_\Gamma)^*$ in the sense of  \cite[\S2.1]{ChHeMo:17}, i.e., for which the isomorphism \eqref{eq:dual2} is isometric. When $H^{\pm 1}(\R^n)$ are equipped with the norms $\|\cdot\|_{H_k^{\pm 1}(\R^n)}$ this is the choice $P=P_k$, where $P_k$ is the orthogonal projection operator
$$
P_k:H^1(\R^n)\to \widetilde H^{1}(\Gamma^c)^{\perp_k}, 
$$
with $\perp_k$ as defined in \S\ref{sec:sub}, so that $R(P)=\widetilde H^{1}(\Gamma^c)^{\perp_k}$.
In the case that $P=P_t$, for some $t>0$, we will denote $A_k$ by $A_{t,k}$, i.e.,
\begin{equation} \label{eq:Akdef}
A_{t,k}\psi:=P_t(\chi \mathcal{A}_k\psi), \qquad \psi \in H^{-1}_\Gamma, \quad t,k>0.
\end{equation}

In \cite{caetano2025integral} the spaces $H^{\pm 1}(\R^n)$ are equipped with their usual norms \eqref{eq:unorm}, i.e., the norms $\|\cdot\|_{H_1^{\pm 1}(\R^n)}$, and only the choice $P=P_{1}$ is considered, in which case $A_k=A_{1,k}$. This will be one choice that will be our focus. The other choice that will be our focus is the choice $P=P_k$, so that $A_k=A_{k,k}$ and $R(P)$ is a unitary realisation of $(H^{-1}_\Gamma)^*$ when $H^{\pm 1}(\R^n)$ are equipped with the wavenumber dependent norms $\|\cdot\|_{H_k^{\pm 1}(\R^n)}$. Importantly, with these choices $A_k=A_{1,k}$ and $A_k=A_{k,k}$, the norms of $A_k$ and $A_k^{-1}$ coincide with the continuity constant and the reciprocal of the inf-sup constant of $a_k(\cdot,\cdot)$ when $H^{\pm 1}(\R^n)$ is equipped with the norm $\|\cdot\|_{H^{\pm 1}_t(\R^n)}$, for $t=1,k$, respectively.  Indeed, for $t,k>0$,
\begin{align} \label{eq:samenorms}
\|A_{t,k}\|_t &= \sup_{0\neq\varphi\in H_\Gamma^{-1}}  \sup_{0\neq \psi\in H_\Gamma^{-1}} \frac{|a_k(\varphi,\psi)|}{\|\varphi\|_{H^{-1}_t(\R^n)}  \|\psi\|_{H^{-1}_t(\R^n)}},\\ \label{eq:samenorms2}
\|A_{t,k}^{-1}\|^{-1}_t &= \inf_{0\neq\varphi\in H_\Gamma^{-1}}  \sup_{0\neq \psi\in H_\Gamma^{-1}} \frac{|a_k(\varphi,\psi)|}{\|\varphi\|_{H^{-1}_t(\R^n)}  \|\psi\|_{H^{-1}_t(\R^n)}},
\end{align}
with the second equation holding whenever $A_k$ is invertible. Recall that the quantities on the right hand sides of  \eqref{eq:samenorms} and \eqref{eq:samenorms2} are the continuity and inf-sup constants, respectively, and $\|A_{t,k}\|_t$ and $\|A_{t,k}^{-1}\|_t$ are our abbreviations for the norms of $A_{t,k}:H_\Gamma^{-1}\to \widetilde H^{1}(\Gamma^c)^{\perp_t}$ and $A_{t,k}^{-1}: \widetilde H^{1}(\Gamma^c)^{\perp_t}\to H_\Gamma^{-1}$ when $H^{\pm 1}(\R^n)$ are equipped with the norms $\|\cdot\|_{H^{\pm 1}_t(\R^n)}$, i.e.,
\begin{equation} \label{eq:normdef}
\|A_{t,k}\|_t  := \sup_{0\neq\varphi\in H_\Gamma^{-1}}  \frac{\|A_{t,k}\varphi\|_{H^1_t(\R^n)}}{\|\varphi\|_{H^{-1}_t(\R^n)}},  \; \|A_{t,k}^{-1}\|_t  := \sup_{0\neq\varphi\in \widetilde H^1(\Gamma^c)^{\perp_t}}  \frac{\|A^{-1}_{t,k}\varphi\|_{H^{-1}_t(\R^n)}}{\|\varphi\|_{H^{1}_t(\R^n)}}.
\end{equation}

The identities \eqref{eq:samenorms} and \eqref{eq:samenorms2} are an immediate consequence of the fact that, when $P=P_t$, the isomorphism  \eqref{eq:dual2} is isometric and  $\widetilde P$ in \eqref{eq:dual2} is the identity. For example, we have, whenever $A_k$ is invertible, that
\begin{align*}
\|A_{t,k}^{-1}\|_t &= \inf_{0\neq\varphi\in H_\Gamma^{-1}}  \frac{\|A_{t,k}\varphi\|_{H^1_t(\R^n)}}{\|\varphi\|_{H^{-1}_t(\R^n)}}\\
& = \inf_{0\neq\varphi\in H_\Gamma^{-1}} \sup_{0\neq\psi\in H_\Gamma^{-1}}  \frac{|\langle A_{t,k}\varphi,\psi \rangle|}{\|\varphi\|_{H^{-1}_t(\R^n)}\|\psi\|_{H^{-1}_t(\R^n)}}\\
& = \inf_{0\neq\varphi\in H_\Gamma^{-1}} \sup_{0\neq\psi\in H_\Gamma^{-1}}  \frac{|a_k(\varphi,\psi)|}{\|\varphi\|_{H^{-1}_t(\R^n)}\|\psi\|_{H^{-1}_t(\R^n)}},
\end{align*}
where we use that \eqref{eq:dual2} is isometric and  $\widetilde P$ in \eqref{eq:dual2} is the identity to obtain the second line of the above computation.

\section{Bounding $\|A_k\|$ and $\|A_k^{-1}\|$} \label{sec:main1}

In this section we prove the main results of the paper, the bounds on $\|A_{k,k}\|_k$ and $\|A_{k,k}^{-1}\|_k$ that are stated in \S\ref{sec:main} as Propositions \ref{prop:Ak} and \ref{prop:Ak2} (the bounds on $\|A_{k,k}\|_k$) and Theorem \ref{thm:main} and Proposition \ref{prop:lower} (the bounds on $\|A_{k,k}^{-1}\|_k$). Recall that $\|A_{k,k}\|_k$ and $\|A_{k,k}^{-1}\|_k$ are our notations \eqref{eq:Aknorm} for the norms of the operator $A_{k,k}:H^{-1}_\Gamma\to \widetilde H^1(\Gamma^c)^{\perp_k}$, given by \eqref{eq:Akdef}, and its inverse when $H^{-1}_\Gamma \subset H^{-1}(\R^n)$ and $\widetilde H^1(\Gamma^c)^{\perp_k} \subset H^1(\R^n)$ are equipped with the wavenumber-explicit norms \eqref{eq:knorm}. If required, bounds on $\|A_{1,k}\|$ and $\|A_{1,k}^{-1}\|$, the norms of $A_{1,k}:H^{-1}_\Gamma \to \widetilde H^1(\Gamma_c)^\perp$, given by \eqref{eq:Akdef}, and its inverse, when  $H^{-1}(\R^n)$ and $H^1(\R^n)$ are equipped with their usual norms \eqref{eq:unorm}, can be deduced immediately from \eqref{eq:equiv2} and \eqref{eq:equiv2a}.

\subsection{Bounds on $\|A_{k,k}\|_k$} \label{sec:Akbound}

We prove first the bounds in \S\ref{sec:main} on $\|A_{k,k}\|_k$ that are Propositions \ref{prop:Ak} (the upper bound) and \ref{prop:Ak2} (the lower bound).

\begin{proof}[Proof of Proposition \ref{prop:Ak}.]
$A_{k,k} = P_k(\chi\cA_k\psi) =  P_k(\chi\cA_k\chi \psi)$, for $\psi\in H^{-1}_\Gamma$ and $\chi\in C^{\infty}_{0,\Gamma}$,  so that $\|A_{k,k}\|_k\leq$ $\|\chi \cA_k\chi\|_k:=\|\chi\cA_k\chi\|_{H^{-1}_k(\R^n)\to H^{1}_k(\R^n)}$, since $\|P_k\|=1$. Let us fix a $k_0>0$ and a choice of $\chi$, with $0\leq \chi\leq 1$. Noting that $\chi\cA_k\chi:H^s_k(\R^n)\to H^{s+2}_k(\R^n)$ and is bounded for all $s\in \R$ and $k>0$ by \eqref{eq:mapping} and \eqref{eq:equiv}, let us first bound $\chi\cA_k\chi$ as a mapping $L^2(\R^n)\to H^2_k(\R^n)$ for $k\geq k_0$. 

Suppose $f\in L^2(\R^n)$ and $R>0$ is such that $\supp(\chi)\subset B_R$, and let $\phi:=\cA_k \chi f\in H^{2,\mathrm{loc}}(\R^n)$, by \eqref{eq:mapping}. Then there exists $c>0$ such that $\|\chi\phi\|_{H_k^2(\R^n)} \leq c\|\phi\|_{H_k^2(B_R)}$ and $\|\chi f\|_{L^2(\R^n)}\leq \|f\|_{L^2(B_R)}$, for $k\geq k_0$ and $f\in L^2(\R^n)$. Further, from \eqref{eq:inv} and the discussion above Lemma \ref{lem:map}, $(\Delta+k^2)\phi = - \chi f$ and $\phi$ satisfies the Sommerfeld radiation condition. It follows from standard estimates for the free-space cut-off resolvent (e.g., see \cite[Cor.~2.2]{Spence:23} and its proof, which is claimed to hold only for $n=2,3$ but in fact works for all $n\geq 2$), that, for some $c'>0$, $\|\phi\|_{H_k^2(B_R)}\leq c'k\|\chi f\|_{L^2(B_R)}\leq c'k\| f\|_{L^2(\R^n)}$,  for $k\geq k_0$ and $f\in L^2(\R^n)$. Thus, there exists $C>0$ such that
$$
\|\chi\cA_k\chi\|_{L^2(\R^n)\to H^2_k(\R^n)} \leq Ck,  \qquad k\geq k_0. 
$$

 Now it follows easily from \eqref{eq:Newt} that
$$
\langle \chi \cA_k\chi u, v\rangle = \langle u,\overline{\chi \cA_k\chi \bar{v}}\rangle, \qquad u,v\in L^2(\R^n),
$$
in other words that $\chi\cA_k\chi$ is self-adjoint with respect to the  {\em real} inner product $\langle\cdot,\cdot\rangle^r$ on $L^2(\R^n)$, defined by  $\langle u,v\rangle^r=\int_{\R^n}uv\, \rd x$, $u,v\in L^2(\R^n)$.  This implies (cf.~the discussion in  \cite[\S2.3]{chandler2020high}) that, for $k>0$, $\|\chi\cA_k\chi\|_{L^2(\R^n)\to H^2_k(\R^n)} =$ $\|\chi\cA_k\chi\|_{H_k^{-2}(\R^n)\to L^2(\R^n)}$. Thus
\begin{equation} \label{eq:boundn}
\|\chi\cA_k\chi\|_{H_k^{s-1}(\R^n)\to H^{s+1}_k(\R^n)} \leq Ck,  \qquad k\geq k_0,
\end{equation}
for $s=\pm1$. Since $\{H_k^s(\R^n):s\in \R\}$ is an exact interpolation scale for every $k>0$ (see \S\ref{sec:fsRn}), it follows by interpolation that \eqref{eq:boundn} holds for $-1\leq s\leq 1$, in particular that
$$
\|A_{k,k}\|_k\leq \|\chi\cA_k\chi\|_k = \|\chi\cA_k\chi\|_{H^{-1}_k(\R^n)\to H^{1}_k(\R^n)} \leq Ck,  \qquad k\geq k_0.
$$
\end{proof}

The following proof adapts an argument, using a quasi-mode $u_k$ for the Helmholtz equation, which is the standard proof of the lower bound \eqref{eq:nt2}  (see, e.g., the discussion before \cite[Lem.~3.10]{CWMonk2008}).

\begin{proof}[Proof of Proposition \ref{prop:Ak2}.] If $\mathrm{int}(\Gamma)$ is non-empty then there exists a real-valued $\sigma\in C_0^\infty(\R^n)$, not identically zero, with $\supp(\sigma)\subset \mathrm{int}(\Gamma)$. For $k>0$ define $u_k\in C_0^\infty(\R^n)$ by $u_k(x) := \re^{\ri k x_1} \sigma(x)$, $x\in \R^n$. Then, where $\phi_k:=-(\Delta +k^2) u_k\in C_0^\infty(\R^n)\cap H^{-1}_\Gamma$, $\phi_k(x) = -(2\ri k \partial_{x_1}\sigma(x) +\Delta\sigma(x))\re^{\ri k x_1}$, $x\in \R^n$. Further, where $\chi\in C^{0,\infty}_\Gamma$, using \eqref{eq:inv},
$$
A_{k,k}\phi_k = P_k(\chi\cA_k\phi_k) = P_k(\chi u_k) = P_k(u_k)=u_k,
$$
where the last equality holds since $u_k$ is supported in $\Gamma$, so that $u_k\in (\widetilde H^1(\Gamma^c))^{\perp_k}$. Thus
$$
\|A_{k,k}\|_k \geq \|u_k\|_{H^1_k(\R^n)}/\|\phi_k\|_{H^{-1}_k(\R^n)}, \qquad k>0.
$$
Now $\|u_k\|_{H^1_k(\R^n)}\geq k\|u_k\|_{L^2(\R^n)} = k\|\sigma\|_{L^2(\R^n)}$. Further, for $v\in H^1_k(\R^n)$,
\begin{eqnarray*}
|\langle \phi_k,v\rangle| &=& |(\phi_k,v)_{L^2(\R^n)}| \leq \|2\ri k \partial_{x_1}\sigma + \Delta \sigma\|_{L^2(\R^n)}\|v\|_{L^2(\R^n)}\\
& \leq & (\|\partial_{x_1}\sigma\|_{L^2(\R^n)} + k^{-1}\|\Delta \sigma\|_{L^2(\R^n)})\|v\|_{H^1_k(\R^n)},
\end{eqnarray*}
so that $\|\phi\|_{H^{-1}(\R^n)} \leq \|\partial_{x_1}\sigma\|_{L^2(\R^n)} + k^{-1}\|\Delta \sigma\|_{L^2(\R^n)}$. Thus
$$
\|A_{k,k}\|_k \geq \frac{k\|\sigma\|_{L^2(\R^n)}}{\|\partial_{x_1}\sigma\|_{L^2(\R^n)} + k^{-1}\|\Delta \sigma\|_{L^2(\R^n)}}, \qquad k>0,
$$
from which the result follows.
\end{proof}

\subsection{Bounds on $\|A_{k,k}^{-1}\|_k$} \label{sec:Ainv}

In this section we prove our main bound, Theorem \ref{thm:main}, on $\|A_{k,k}^{-1}\|_k$. Our proof is, in part, inspired by arguments made to bound the Dirichlet to Neumann (DtN) operator for $\Omega$ when $\Omega$ is Lipschitz in Spence \cite[\S3]{spence2014wavenumber} (and see \cite[Lemma 4.2]{chandler2020high}). 
In particular, the idea of using properties of solution operators when the wavenumber is complex as a tool to reduce operator bounds to estimating cut-off resolvent norms is taken from  \cite[\S3]{spence2014wavenumber}.
However, the application here is substantially different, and key elements of the argument in \cite[\S3]{spence2014wavenumber} that use Rellich-identity methods are not applicable in our setting where $\Omega$ is not necessarily Lipschitz.

As the above suggests, we will make use of the Newtonian potential \eqref{eq:Newt} and the integral equation \eqref{eq:iemain} in the case that $k$ is complex, precisely in the case that 
$\mathrm{Im}\, k>0$. The theory is significantly simpler in that case since $k^2$ is not in the spectrum, $[0,\infty)$, of  $-\Delta:L^2(\R^n)\to L^2(\R^n)$, with domain $\{v\in H^1(\R^n):\Delta v\in L^2(\R^n)\}$, the self-adjoint free-space Laplacian. Indeed, $-\Delta - k^2$ is an isomorphism from $H^{s+2}(\R^n)\to H^s(\R^n)$, for all $s\in \R$, and $\cA_k = (-\Delta-k^2)^{-1}:H^s(\R^n)\to H^{s+2}(\R^n)$ is its inverse, with a convolution kernel \eqref{eq:Phidef} that is exponentially decreasing at infinity (see, e.g,  \cite[p.~282]{mclean2000strongly}). Proposition \ref{prop:ie1} holds for $\mathrm{Im}\, k>0$ with some simplifications: for $\phi\in H^{-1}_\Gamma$ (indeed, for any $\phi\in H^{-1}(\R^n)$), $u=\cA_k\phi\in H^1(\R^n)$ so that the $\chi$ in \eqref{eq:Akdef2} is superfluous and $A_k = P\cA_k$. Corollary \ref{cor:exists} also applies in this setting. So does Proposition \ref{prop:coer}, with the simplification that $a_k(\cdot,\cdot)$ is coercive in the case that $\mathrm{Im}\, k>0$, so that (cf.~Theorem \ref{thm:invert}) $A_k$ is invertible by Lax-Milgram for $\mathrm{Im}\, k>0$. The coercivity of $a_k(\cdot,\cdot)$, i.e.~of $A_k=P \cA_k$, follows from the coercivity of $\cA_k:H^{-1}_k(\R^n)\to H^1(\R^n)=((H^{-1}(\R^n))^*$  (cf.~the deduction that $A_k^0$ is coercive because $\cA_k^0$ is coercive, in the proof of Proposition \ref{prop:coer}). 

To prove Theorem \ref{thm:main} we need to show that, given any $k_0>0$ and $R>R_\Gamma$, there exists $C>0$ such that, if $g\in \widetilde H^1(\Gamma^c)^\perp$ and $k\in [k_0,\infty)\setminus \Sigma(\Omega_-)$ (so $A_k$ is invertible by Theorem \ref{thm:invert}), and if $\phi:=A_{k,k}^{-1}g\in H^{-1}_\Gamma$, then
\begin{equation} \label{eq:need}
\|\phi\|_{H_k^{-1}(\R^n)} \leq Ck^2\left(C_{k,R}(\Omega) + C_k(\Omega_-)\right)\|g\|_{H^1_k(\R^n)}.
\end{equation}
Our argument to show that the above bound holds, has the following steps, captured in Lemmas \ref{lem:step1}-\ref{lem:step3} below. These lemmas, making the choice $\beta(k)=k^p$, for some $p\in [1,2]$, in steps \ref{step2} and \ref{step3}, provide the proof of Theorem \ref{thm:main}, also given below.
\begin{enumerate}
\item \label{step1} Bound $\phi\in H^{-1}_\Gamma$ in terms of $u=\cA_k \phi$, for $k\geq k_0$.
\item \label{step2} Bound $w=\cA_\lambda \mu$ in terms of $g\in \widetilde H^1(\Gamma^c)^\perp$, where $\mu \in H^{-1}_\Gamma$ is the solution of $A_\lambda \mu = g$, in the case that $\mathrm{Im}\, \lambda>0$ and $\lambda^2 = k^2 + \ri \beta(k)$, with $0<\beta(k)\leq k^2$.
\item \label{step3} For $k\in [k_0,\infty)\setminus \Sigma(\Omega_-)$ and $R>R_\Gamma$, choose $\psi \in C_{0,\Gamma}^\infty$ with $\supp(\psi)\subset B_R$, and define  $u$ and $w$ as in steps \ref{step1} and \ref{step2}, with $\phi$ the unique solution of $A_{k,k}\phi = g$. Then it is easy to see that $v:= u-\psi w\in \widetilde H^1(\Gamma^c)$, $v$ satisfies the radiation condition \eqref{eq:src}, and $\Delta v + k^2 v = -h\in L^2(\R^n)$, with $\supp(h)\subset B_R$ and $h$ given explicitly in terms of $w$, which leads to a bound for $v$ in terms of $w$, $C_{k,R}(\Omega)$, and $C_k(\Omega_-)$.
\end{enumerate}

\begin{lemma} \label{lem:step1} Given any $k_0>0$ and $R>R_\Gamma$ there exists $C_1>0$ such that
$$
\|\phi\|_{H_k^{-1}(\R^n)} \leq C_1 \|u\|_{H^1_k(B_R)},
$$
if $\phi\in H^{-1}_\Gamma$, $k\geq k_0$, and $u:= \cA_\lambda \phi$, where $\lambda\in \C$ with $\mathrm{Im}\,\lambda \geq 0$ and $\lambda^2=k^2+\ri \beta(k)$, with $0\leq \beta(k)\leq k^2$. 
\end{lemma}
\begin{proof} Suppose that $\phi\in H^{-1}_\Gamma$ and $k\geq k_0$, set $u:= \cA_\lambda \phi$, with $\lambda$ as in the statement of the lemma, choose $\psi\in C_{0,\Gamma}^\infty$ with $\supp(\psi)\subset B_R$, and set $\tilde u := \psi u$ and 
\begin{equation} \label{eq:psi}
C_\psi := \max\{\|\psi\|_{L^\infty(\R^n)},\| |\nabla\psi| \|_{L^\infty(\R^n)},\|\Delta \psi\|_{L^\infty(\R^n)}\}.
\end{equation}
Then, since $\Delta u + \lambda^2 u = -\phi$ (by \eqref{eq:inv} in the case $\beta(k)=0$) it follows that $\Delta \tilde u + \lambda^2 \tilde u = -\psi\phi + 2\nabla u \cdot \nabla \psi + u\Delta \psi$. Since $\psi\phi=\phi$, $|\lambda^2| \leq 2k^2$, $\|v\|_{H_k^{-1}(\R^n)}\leq k^{-1}\|v\|_{L^2(\R^n)}$, for $v\in L^2(\R^n)$, and 
$$
\|\Delta v\|^2_{H_k^{-1}(\R^n)}\leq \int_{\R^n} |\hat v(\xi)|^2|\xi|^2\, \rd \xi = \| |\nabla v| \|^2_{L^2(\R^n)}, \qquad v\in H^1(\R^n), 
$$
it follows that
\begin{align*}
\|\phi\|_{H_k^{-1}(\R^n)} &= \|2\nabla u\cdot \nabla \psi + u\Delta \psi -\lambda^2\psi u-\Delta \tilde u\|_{H^{-1}_k(\R^n)}\\
& \leq  2k^{-1} C_\psi \| |\nabla u| \|_{L^2(B_R)} + (k^{-1} +2k)C_\psi \|u\|_{L^2(B_R)} + \| |\nabla \tilde u| \|_{L^2(\R^n)}\\
& \leq  (1+2k^{-1}) C_\psi \| |\nabla u| \|_{L^2(B_R)} + (k^{-1} +1 + 2k)C_\psi \|u\|_{L^2(B_R)},
\end{align*}
from which the claimed bound follows.
\end{proof}

In the following lemma and subsequently we define $A_{k,\lambda}$, for $k>0$ and $\mathrm{Im}\, \lambda>0$, by \eqref{eq:Akdef}, i.e.\ by $A_{k,\lambda}\phi := P_k(\chi \cA_\lambda)\phi = P_k\cA_\lambda \phi$, $\phi\in H_\Gamma^{-1}$.
\begin{lemma} \label{lem:step2} Suppose that  $k_0>0$. Then  there exists $C_2>0$ such that, for all  $k\geq k_0$, $g\in \widetilde H^1(\Gamma^c)^{\perp_k}$, and 
$\lambda\in \C$ with $\lambda_i := \mathrm{Im}\, \lambda >0$ and $\lambda^2=k^2+\ri \beta(k)$, with $0<\beta(k)\leq k^2$,  it holds that
$$
\|w\|_{H^1_k(\R^n)} \leq \frac{C_2k^2}{\beta(k)} \|g\|_{H^1_k(\R^n)},
$$ 
where $w:= \cA_\lambda \mu$ and $\mu:= A_{k,\lambda}^{-1}g\in H^{-1}_\Gamma$.
\end{lemma}
\begin{proof} Suppose that $k\geq k_0$ and $\lambda$, $g$, $\mu$, and $w$ are as in the statement of the lemma. Let us define $b_\lambda: H_k^1(\mathbb{R}^n)\times H_k^1(\mathbb{R}^n)\rightarrow\mathbb{C}$ by
\begin{equation*}
 b_\lambda(u,v):=\int_{\R^n}\left(\lambda^2 u \bar v - \nabla u\cdot \nabla \bar v\right)\, \rd x, \qquad u,v \in H^1(\R^n).
\end{equation*}
This is the sesquilinear form that appears in the weak form of the Helmholtz equation on $\R^n$: 
if $u\in H^1(\R^n)$ and $f\in H^{-1}(\R^n)$, then $\Delta u + \lambda^2 u = -f$ in a distributional sense if and only if
$$
b_\lambda(u,v) = -\langle f,v\rangle, \qquad \forall v\in H^1(\R^n).
$$
In particular, since $(\Delta + \lambda^2)w = - \mu$ and $\langle \mu, v\rangle =0$ for all $v\in \widetilde H^1(\Gamma^c)$, we have that
\begin{align*}
|b_\lambda(w,w)| & = |\langle \mu,\cA_\lambda \mu\rangle| = |\langle \mu, A_{k,\lambda} \mu\rangle|=|\langle \mu,g\rangle|\\ & \leq \|\mu\|_{H_k^{-1}(\R^n)} \|g\|_{H^1_k(\R^n)} \leq C_1\|w\|_{H_k^{1}(\R^n)} \|g\|_{H^1_k(\R^n)},
\end{align*}
by Lemma \ref{lem:step1}. Further, for $u\in H^1(\R^n)$,
$$
\bar \lambda b_\lambda(u,u) = \int_{\R^n}\left(\lambda |\lambda|^2 |u|^2  - \bar\lambda |\nabla u|^2\right)\, \rd x
$$
so that
$$
|\bar \lambda b_\lambda(u,u)| \geq \mathrm{Im}( \bar \lambda b_\lambda(u,u)) = \lambda_i \int_{\R^n}\left(|\lambda|^2 |u|^2  + |\nabla u|^2\right)\, \rd x \geq \lambda_i \|u\|_{H^1_k(\R^n)}^2,
$$
i.e., $b_\lambda(\cdot,\cdot)$ is coercive, and
$$
\lambda_i \|w\|_{H_k^{1}(\R^n)} \leq C_1|\lambda| \|g\|_{H^1_k(\R^n)}.
$$
Further, $\beta(k)\leq k^2$ so that $|\lambda| \leq \sqrt{2}\,k$ and
$$
\lambda_i = \frac{\beta(k)}{\sqrt{2}\sqrt{k^2+\sqrt{k^4+\beta^2(k)}}}\geq \frac{\beta(k)}{\sqrt{2}\sqrt{1+\sqrt{2}}\,k},
$$
from which the claimed result follows.
\end{proof}

\begin{lemma} \label{lem:step3} Suppose that $k_0>0$ and $R>R_\Gamma$ and choose $\psi\in C_{0,\Gamma}^\infty$ with $\supp(\psi)\subset B_R$. Then there exists $C_3>0$ such that, if  $k\in [k_0,\infty)\setminus \Sigma(\Omega_-)$, $g\in \widetilde H^1(\Gamma^c)^{\perp_k}$, $\lambda\in \C$, $\lambda_i := \mathrm{Im}\, \lambda >0$, $\lambda^2=k^2+\ri \beta(k)$, with $0<\beta(k)\leq k^2$, $\phi:=A_{k,k}^{-1}g$, $\mu:= A_{k,\lambda}^{-1}g$, $u:= \cA_k\phi$, $w:= \cA_\lambda \mu$, and $v:= u-\psi w$, then 
\begin{equation} \label{eq:lem3}
\|v\|_{H^1_k(B_R)} \leq C_3 (k+\beta(k))\left(C_{k,R}(\Omega)+C_k(\Omega_-)\right) \|w\|_{H^1_k(B_R)}.
\end{equation}
\end{lemma}
\begin{proof}
We note that $u\in \Hol(\R^n)$ and $w\in H^1(\R^n)$, so that $v\in \Hol(\R^n)$ and $v|_{\Omega_-}\in H^1(\Omega_-)$, and that $v$ satisfies the radiation condition \eqref{eq:src} since $u$ does. Further, for every $\chi\in C_{0,\Gamma}^\infty$, $P_k(\chi v) = A_{k,k}\phi-A_{k,\lambda}\mu =0$, so that $\chi v\in \widetilde H^1(\Gamma^c)$. Thus $(\chi v)|_\Omega \in H_{0}^1(\Omega)$, for all $\chi\in C_{0,\Gamma}^\infty$, and hence for all $\chi\in C_0^\infty(\R^n)$, so that $v|_\Omega \in \Hol_0(\Omega)$. Moreover $(\Delta+k^2)v = -h$ in $\Gamma^c$, where
$$
h := w\Delta\psi+2\nabla\psi\cdot\nabla w-\ri \beta(k)\psi w \in L^2_{\mathrm{comp}}(\R^n),
$$
with $\supp(h)\subset B_R$.  Where  $C'_{k,R}(\Omega):=$ $\|\chi_R R(k;\Omega)\chi_R\|_{L^2(\Omega_R)\to H^1_k(\Omega_R)}$, it follows that $\|v\|_{H^1_k(\Omega_R)} \leq C'_{k,R}(\Omega)\|h\|_{L^2(\Omega_R)}$ and, if $\Omega_-$ is non-empty, that $\|v\|_{H^1_k(\Omega_-)} \leq C'_k(\Omega_-)\|h\|_{L^2(\Omega_-)}$, where
 $C'_k(\Omega_-) := \|R(k;\Omega_-)\|_{L^2(\Omega_-)\to H^1_k(\Omega_-)}$. Defining $C'_k(\Omega_-) := 0$ in the case that $\Omega_-=\emptyset$, it follows that 
$$
\|v\|_{H^1_k(B_R)} \leq \left(C'_k(\Omega_-)+C'_{k,R}(\Omega)\right) \|h\|_{L^2(B_R)}.
$$
Further, by \cite[Equation (10)]{SiavashSimon0}, where $C_{k,R}(\Omega)$ is defined by \eqref{eq:CkR2},
$$
C'_{k,R}(\Omega) \leq \sqrt{2k^2 (C_{k,R}(\Omega))^2 + C_{k,R}(\Omega)}.
$$
Arguing similarly, recalling \eqref{eq:resol},
$$
C'_{k}(\Omega_-) \leq \sqrt{2k^2 (C_{k}(\Omega_-))^2 + C_{k}(\Omega_-)}.
$$
Moreover, where $C_\psi$ is given by \eqref{eq:psi},
\begin{align*}
\|h\|_{L^2(B_R)} &\leq C_\psi (1+\beta(k)) \|w\|_{L^2(B_R)} + 2C_\psi \| |\nabla w|\|_{L^2(B_R)}\\ & \leq C_\psi (k^{-1}(1+\beta(k))+2) \|w\|_{H^1_k(B_R)}.
\end{align*}
The claimed result follows by combining these inequalities.
\end{proof}

\begin{proof}[Proof of Theorem \ref{thm:main}]
As noted above we need to show that, given any $k_0>0$ and $R>R_\Gamma$, there exists $C>0$ such that, if  $k\in [k_0,\infty)\setminus \Sigma(\Omega_-)$ and $g\in \widetilde H^1(\Gamma^c)^{\perp_k}$, and if $\phi:=A_{k,k}^{-1}g$, then \eqref{eq:need} holds. So suppose $k_0>0$, $R>R_\Gamma$,  $k\in [k_0,\infty)\setminus \Sigma(\Omega_-)$, $g\in \widetilde H^1(\Gamma^c)^{\perp_k}$, and set $\phi:=A_{k,k}^{-1}g\in H^{-1}_\Gamma$. By Lemma \ref{lem:step1}, applied with $\beta(k)=0$, $\|\phi\|_{H_k^{-1}(\R^n)}\leq C_1 \|u\|_{H_k^1(B_R)}$, where $u:= \cA_k \phi$. Define $\psi$, $\mu$, $w$, and $v$ as in Lemma \ref{lem:step3}. Then \eqref{eq:lem3} holds by Lemma \ref{lem:step3}, so that, where $C_\psi$ is given by \eqref{eq:psi},
\begin{align*}
\|u\|_{H^1_k(B_R)} &\leq \|v\|_{H^1_k(B_R)} + \|\psi w\|_{H^1_k(B_R)} \leq  \|v\|_{H^1_k(B_R)} + 2C_\psi \|w\|_{H^1_k(B_R)}\\
&\leq \left(C_3 (k+\beta(k))\left(C_{k,R}(\Omega)+C_k(\Omega_-)\right) + 2C_\psi\right)\|w\|_{H^1_k(B_R)}\\
& \leq  \left(C_3 (k+\beta(k))\left(C_{k,R}(\Omega)+C_k(\Omega_-)\right) + 2C_\psi\right)\frac{C_2k^2}{\beta(k)} \|g\|_{H^1_k(\R^n)},
\end{align*}
by Lemma \ref{lem:step2}. Recalling the bound \eqref{eq:nt2}, the result follows by choosing $\beta(k)=k^{2-p}_0 k^p$, for some $p\in [1,2]$. 
\end{proof}

\begin{remark} \label{rem:spence}
We have noted above that our arguments in the section are motivated by arguments in Spence \cite[\S3]{spence2014wavenumber} (and see \cite[Lemma 4.2]{chandler2020high}), bounding the DtN operator in the case when $\Omega$ is Lipschitz, in which arguments a complex wavenumber $\lambda$ satisfying $\lambda^2 = k^2+\ri \beta(k)$ plays a role, as it does above. But, in contrast to  \cite[\S3]{spence2014wavenumber}, where the choice $\beta(k) = k$ gives a sharper estimate than the choice $\beta(k)=k^2$, there appears to be no similar sharpening of the argument in our case by the choice $\beta(k)=k_0k$.
\end{remark}

Our proof of the next proposition has something of the flavour of the proof of \cite[Theorem 2.8]{BetCha11} (and see \cite[pp.~222--223]{chandler2012numerical}).
\begin{proof}[Proof of Proposition \ref{prop:lower}] We have seen in the proof of Proposition \ref{prop:Ak} in \S\ref{sec:Akbound} that, for every $\chi\in C_0^\infty(\R^n)$ and $k_0>0$, there exists $C>0$ such that \eqref{eq:boundn} holds for $-1\leq s\leq 1$. For all $R>0$,  choosing $\chi\in C^\infty_{0}(\R^n)$ so that $\chi=1$ on $B_R$ and applying this bound with $s=-1$ and $s=0$, it follows that there exists $C_R>0$, depending only on $\chi$, $R$ and $k_0>0$, such that
\begin{equation} \label{eq:Akbounds}
\|\cA_k\|_{H_\Gamma^{-1}\to L^2(B_R)} \leq C_R, \quad \|\cA_k\|_{L^2(B_R)\to L^2(B_R)} \leq \frac{C_R}{k}, 
\end{equation}
and
\begin{equation} \label{eq:Akbounds2}
\|\chi \cA_k\|_{L^2(B_R)\to H_k^1(\R^n)} \leq C_R,
\end{equation}
for $k\geq k_0$, where $H_\Gamma^{-1}$ is equipped with the norm $\|\cdot\|_{H^{-1}_k(\R^n)}$ in the first of these inequalities.

 By \cite[Theorem 1.4]{SiavashSimon0} and an inspection of its proof we see that, given some increasing, unbounded positive sequence $(k_m)_{m\in \N}$  satisfying \eqref{eq:kconstraint}, for some $c>0$, there exists $R_1>0$ such that, for every unbounded, increasing sequence $(\tilde a_m)_{m\in \N}\subset (0,\infty)$, there exists a compact obstacle $O_1$, with $\Omega_1:= \R^n\setminus O_1$ connected and $\max_{x\in O_1}|x|\leq R_1$, such that, for every $m\in \N$ and $R>R_1$,
$
C_{k_m,R}(\Omega_1) \geq \tilde a_m.
$
Let $O_2:= \overline{B_{R_1}} + 3R_1e_1$, where $e_1\in \R^n$ is the unit vector in the $x_1$ direction, so that $4R_1\geq |x| \geq  2R_1$, for $x\in O_2$. Then, where  $\Gamma := O:= O_1\cup O_2$, $\Omega := \R^n\setminus O$ is connected. Further it is easy to see, by inspecting the argument of \cite[Theorem 1.4]{SiavashSimon0}, that also
$ C_{k_m,R}(\Omega)\geq \tilde a_m$ for each $m\in \N$ and $R>R_\Gamma := \max_{x\in \Gamma}|x| = 4R_1$. Further,  since $\mathrm{int}(\Gamma)\supset B_{R_1}+3R_1e_1$ is non-empty, it follows, by Proposition \ref{prop:Ak2} and an inspection of its proof, that, for every $k_0>0$ there exists $c'>0$, depending only on $k_0$ and $R_1$, such that $\|A_{k,k}\|_k\geq c'k$, for $k\geq k_0$. 

So suppose $(k_m)_{m\in \N}$  is a positive, unbounded sequence satisfying \eqref{eq:kconstraint}, for some $c>0$,  set $k_0 := k_1$, let $R_1>0$ and $c'>0$, dependent only on $k_0$ and $R_1$, be as in the above paragraph,  choose $R>4R_1$, $\chi\in C_0^\infty(\R^n)$ with $\chi=1$ on $B_R$, and $C_R>0$ such that \eqref{eq:Akbounds} and \eqref{eq:Akbounds2} hold, and let $(a_m)_{m\in \N}$ be some unbounded, increasing, positive sequence. Define 
$$
\tilde a_m := \max\left(1,\frac{1}{c'k_0}\right)(C_R)^2a_m + \frac{C_R}{k_0}, \qquad m\in \N,
$$ 
and choose a compact set $O_1\subset \R^n$ with $\max_{x\in O_1}|x|\leq R_1$ such that $C_{k_m,R}(O_1^c)\geq \tilde a_m$ for each $m\in \N$, and define $O_2$, $\Gamma:=O:=O_1\cup O_2$, and $\Omega:= O^c$  as  in the above paragraph, so that $\Omega$ is connected, $R_\Gamma:= \max_{x\in \Gamma}|x|=4R_1<R$, $\chi\in C_{0,\Gamma}^\infty$, and $C_{k_m,R}(\Omega)\geq \tilde a_m$ for each $m\in \N$. Note that $A_{k,k}$ is invertible for every $k>0$, in particular for $k=k_m$, $m\in \N$, by Theorem \ref{thm:invert}.

For each $m\in \N$, since $C_{k_m,R}(\Omega)\geq \tilde a_m$, there exists $f_m\in L^2_{\mathrm{comp}}(\R^n)$ with $\|f_m\|_{L^2(\R^n)}=1$ and $\supp(f_m)\subset \Omega_R$ such that $\tilde u_m:=  R(k_m;\Omega)f_m\in \Hol_0(\Omega)$ satisfies $\|\tilde u_m\|_{L^2(\Omega_R)}\geq a_m$. Let $u_m$ denote $\tilde u_m$ extended to $\R^n$ by zero, and let $v_m:= u_m-\cA_{k_m} f_m\in \Hol(\R^n)$ and $\phi_m := -(\Delta+k_m^2)v_m$. Then, for each $m\in \N$, $\phi_m\in H^{-1,\mathrm{loc}}(\R^n)$ and $\phi_m=0$ in $\Omega$, so that $\phi_m\in H_\Gamma^{-1}$. Defining, for each $m\in \N$, $w_m:= v_m-\cA_{k_m}\phi_m = u_m -\cA_{k_m}(f_m+\phi_m)\in \Hol(\R^n)$, we see, using \eqref{eq:inv}, that $(\Delta+k_m^2)w_m=0$ in $\R^n$, but also $w_m$ satisfies the radiation condition \eqref{eq:src} with $k=k_m$ so that (by uniqueness of solution of the scattering problem in the special case that the obstacle is the empty set), $w_m=0$ so that $v_m=\cA_{k_m}\phi_m$. Thus,
using \eqref{eq:Akbounds},
$$
\|v_m\|_{L^2(\Omega_R)} \leq C_R \|\phi_m\|_{H^{-1}_k(\R^n)} \quad \mbox{and} \quad \|\cA_{k_m} f_m\|_{L^2(\Omega_R)} \leq C_R/k_m,
$$
for $m\in \N$, so that
\begin{align*}
 \|\phi_m\|_{H^{-1}_k(\R^n)} &\geq \frac{\|v_m\|_{L^2(\Omega_R)}}{C_R}\geq \frac{\|u_m\|_{L^2(\Omega_R)}-\|\cA_{k_m}f_m\|_{L^2(\Omega_R)}}{C_R}\\
 & \geq \frac{\tilde a_m}{C_R}-k_0^{-1} \geq \max\left(1,\frac{1}{c'k_0}\right)C_R a_m.
\end{align*}

For $m\in \N$ define $g_m \in \widetilde H^1(\Gamma^c)^{\perp_{k_m}}= \widetilde H^1(\Omega)^{\perp_{k_m}}$ by
\begin{align*} 
g_m := A_{k_m,k_m}\phi_m =P_{k_m}(\chi\cA_{k_m}\phi_m) = P_{k_m}(\chi v_m) = -P_{k_m}(\chi \cA_{k_m}f_m),
\end{align*}
since $\chi u_m\in \widetilde H^1(\Omega)$. Then, for $m\in \N$,
$$
\|g_m\|_{H^1_k(\R^n)} \leq \|\chi \cA_{k_m}f_m\|_{H^1_k(\R^n)} \leq C_R,
$$
by \eqref{eq:Akbounds2}, so that
$$
\|A_{k_m,k_m}^{-1}\|_{k_m} \geq \frac{\|\phi_m\|_{H^{-1}_k(\R^n)}}{\|g_m\|_{H^1_k(\R^n)}} \geq \max\left(1,\frac{1}{c'k_0}\right)a_m\geq a_m,
$$
and also, since $\|A_{k,k}\|_k\geq c'k$, for $k\geq k_0$, 
$$
\cond_{k_m}(A_{k_m,k_m}) = \|A_{k_m,k_m}\|_{k_m} \|A_{k_m,k_m}^{-1}\|_{k_m} \geq c'k_0 \max\left(1,\frac{1}{c'k_0}\right)a_m\geq a_m.
$$
\end{proof}

\section{The case that $\Gamma$ is a $d$-set} \label{sec:dset}
In this section we focus on the case that $\Gamma$ is a $d$-set, in the sense of \eqref{eq:dset} below, 
in which case $\Gamma$ has Hausdorff dimension $d$. Significant examples of $d$-sets are: i) $\Gamma$ is the closure of a bounded Lipschitz domain ($d=n$); ii) $\Gamma$ is the boundary of a bounded Lipschitz domain ($d=n-1$); iii) $\Gamma$ is the attractor of an iterated function system of contracting similarities satisfying the standard open set condition (see, e.g., \cite[\S9.2]{Falconer2014} or \cite[\S2.1]{caetano2025integral}), for example the Sierpinski triangle, the Cantor set or Cantor dust, etc., in which case $\Gamma$ is a $d$-set with $d=\dimH(\Gamma)$, the Hausdorff dimension of $\Gamma$ (see \cite[Theorem 4.7]{Triebel97FracSpec}). 
We will restrict attention to the case $n-2<d\leq n$; if $\Gamma$ is a $d$-set with $0\leq d\leq n-2$ then $H_\Gamma^{-1}=\{0\}$ \cite[Rem.~3.5]{caetano2025integral}, so that \eqref{eq:iemain} has only the trivial solution and the scattered field $u = \cA_k \phi$ is zero.

The significance of the $d$-set case is as follows. Following \cite{caetano2025integral}, in a slight abuse of notation, we have described \eqref{eq:iemain} as an integral equation because the operator $A_k$, given by \eqref{eq:Akdef2}, is a composition of an integral operator $\cA_k$, given by \eqref{eq:Newt}, with two simpler operators (a multiplication and a projection operator). In the case that $\Gamma$ is a $d$-set, for some $n-2<d\leq n$, we can go further and write \eqref{eq:iemain} equivalently as an equation
\begin{equation} \label{eq:ie22}
\bA_k \bphi = \gamma g,
\end{equation}
where $\bA_k$ is itself an integral operator, a generalisation of the standard single-layer BIO $S_k$ (cf.~\eqref{eq:Sk} and \eqref{eq:bAkdef}). Further, we can relate the norms of $\bA_k$ and $A_k$ and their inverses, which we do, as the main results of this section, in \eqref{eq:bAknorm} and \eqref{eq:bAknorm2} below. As we point out, this leads to $k$-explicit bounds on the norms of $\bA_k$ and its inverse, by application of the results of \S\ref{sec:main}.  This leads in turn, see \S\ref{sec:slp2}, to bounds on the classical single-layer BIO and its inverse, by applying \eqref{eq:bAknorm} in the case $d=n-1$.

As in \cite[\S1.1]{JoWa84} and \cite[\S3]{Triebel97FracSpec}, given $0\leq d\leq n$ we say that a closed set $F\subset \R^n$ is a {\em$d$-set} (or an \textit{Ahlfors-David $d$-regular set}) if there exist constants $0<c_1\leq c_2$ such that
\begin{align}
\label{eq:dset}
c_{1}r^{d}\leq \cH^d(F\cap B(x,r))\leq c_{2}r^{d},\qquad x\in{F},\quad0<r\leq1,
\end{align}
where $B(x,r)\subset \R^n$ denotes the closed ball of radius $r$ centred on $x$ and $\cH^d$ denotes Hausdorff measure.  (For convenience we adopt the normalisation of \cite[Def. 2.1]{EvansGariepy}, so that
$\cH^d$ coincides with Lebesgue measure for $d = n$ \cite[Thm. 2.5]{EvansGariepy}.)
Condition \eqref{eq:dset} implies that $F$ is uniformly locally $d$-dimensional in the sense that $\dimH(F\cap B(x,r))=d$ for every $x\in {F}$ and $r>0$. 

	Let $\gamma:\cS(\R^n)\to C(\Gamma)$ denote the trace (or restriction) operator, defined by $\gamma u := u|_\Gamma$, $u\in \cS(\R^n)$. For $d>n-2$ this  trace operator has a unique extension to a bounded linear operator $\gamma:H^1(\R^n)\to L^2(\Gamma,\cH^d)$ (see \cite[Theorem 1]{Jonsson79}), where $L^2(\Gamma,\cH^d)$ denotes the Hilbert space of complex-valued functions defined on $\Gamma$ that are square-integrable with respect to $\cH^d$. Identifying $L^2(\Gamma,\cH^d)$ with its own dual in the standard way, the adjoint of this extension is $\gamma^*:L^2(\Gamma,\cH^d)\to H^{-1}(\R^n)$,  given by
\begin{equation} \label{eq:gamadj}
\langle\gamma^* \bpsi,u\rangle = (\bpsi,\gamma u)_{L^2(\Gamma,\cH^d)}, \qquad \bpsi\in L^2(\Gamma,\cH^d), \quad u\in H^1(\R^n),
\end{equation}
where $(\cdot,\cdot)_{L^2(\Gamma,\cH^d)}$ denotes the inner product on $L^2(\Gamma,\cH^d)$. Following \cite{caetano2025integral}, let $\bH(\Gamma):= \gamma(H^1(\R^n))$ denote the range of $\gamma:H^1(\R^n)\to L^2(\Gamma,\cH^d)$, which we equip with the quotient norm $\|\cdot\|_{\bH(\Gamma)}$, defined by
\begin{equation} \label{eq:bHnormdef}
\|\bpsi\|_{\bH(\Gamma)} := \inf_{\stackrel{u\in H^1(\R^n)}{\gamma u = \bpsi}}\|u\|_{H^1(\R^n)}, \qquad \bpsi \in \bH(\Gamma),
\end{equation}
so that $\bH(\Gamma)$ is a Hilbert space, unitarily isomorphic to the quotient space $H^1(\R^n)/\ker(\gamma)$, with $\bH(\Gamma)$ continuously embedded and dense in $L^2(\Gamma,\cH^d)$. Let $\bH^*(\Gamma)\supset L^2(\Gamma,\cH^d)$ denote the dual space of $\bH(\Gamma)$. Then $(\bH(\Gamma),$ $L^2(\Gamma,\cH^d),$ $\bH^*(\Gamma))$ is a Gelfand triple, with $L^2(\Gamma,\cH^d)$ the pivot space which is dense and continuously embedded in $\bH^*(\Gamma)$. We remark that our notations here simplify those in \cite{caetano2025integral}; $\bH(\Gamma)$ and $\bH^*(\Gamma)$ are denoted, respectively, $\bH^{t_d}(\Gamma)$ and $\bH^{-t_d}(\Gamma)$ in \cite[\S3.2]{caetano2025integral}, where the exponent $t_d$ has the value $t_d=1-(n-d)/2$.

As noted in \S\ref{sec:main}, and discussed in detail in \S\ref{sec:IE}, the operator $A_k$ in \eqref{eq:iemain} is a bounded operator from $H_\Gamma^{-1}$ to $R(P)\subset H^1(\R^n)$, where the closed subspace $R(P)$ is the range of a continuous projection operator $P:H^1(\R^n)\to H^1(\R^n)$ such that $\ker(P) = \widetilde H^1(\Gamma^c)$. As we discussed in \S\ref{sec:IE}, there is a natural isomorphism enabling the identification of $R(P)$ with $(H_\Gamma^{-1})^*$ and so $H_\Gamma^{-1}$ with $(R(P))^*$. Further, the kernel of $\gamma:H^1(\R^n)\to L^2(\Gamma,\cH^d)$ is $\ker(\gamma)=\widetilde H^1(\Gamma^c)$ (see, e.g., \cite[Theorem 3.9]{caetano2025integral} for the case $n-2<d<n$, and \cite[Corollary 3.2]{hinz25} for the case $d=n$). Thus $\gamma:R(P) \to \bH(\Gamma)$ is a continuous bijection, so an isomorphism, and so also is its adjoint $\gamma^*:\bH^*(\Gamma)\to (R(P))^* =H^{-1}_\Gamma$ (cf.~\cite[Theorem 3.9]{caetano2025integral}).

Since $\gamma:R(P) \to \bH(\Gamma)$ and  $\gamma^*:\bH^*(\Gamma)\to (R(P))^*$ are isomorphisms, the operator $\bA_k:\bH^*(\Gamma)\to \bH(\Gamma)$, defined by
\begin{equation} \label{eq:bAkdef}
\bA_k := \gamma A_k \gamma^*,
\end{equation}
is equivalent to $A_k:H_\Gamma^{-1}=(R(P))^*\to R(P)$, given by  \eqref{eq:Akdef2}, and \eqref{eq:ie22} is equivalent to \eqref{eq:iemain}, with the solutions of   \eqref{eq:ie22} and \eqref{eq:iemain} related by $\phi = \gamma^* \bphi$. Importantly, since $R(I-P)=\ker(P)=\widetilde H^1(\Gamma^c)=\ker(\gamma)$, \eqref{eq:bAkdef} and \eqref{eq:Akdef2} imply that
$$
\bA_k = \gamma \chi \cA_k \gamma^*,
$$
making clear that the definition of $\bA_k$ is independent of the choice of $P$.
Further, for $\bpsi\in L^\infty(\Gamma,\cH^d)$, which is a dense subspace of $\bH^*(\Gamma)$,
\begin{equation} \label{eq:bAk2}
\bA_k \bpsi(x) = \int_\Gamma \Phi_k(x,y)\bpsi(y) \rd \cH^d(y), \qquad \mbox{for } \cH^d\mbox{-a.e.} \quad x\in \Gamma,
\end{equation}
(see \cite[Theorem 3.16(iii)]{caetano2025integral}). Clearly, this is a representation for $\bA_k$ as an integral operator with kernel $\Phi_k(\cdot,\cdot)$.

Our focus in this paper, as described in \S\ref{sec:main}, is, firstly, on the canonical choice for $P$ in the case that $H^{\pm1}(\R^n)$ are equipped with the standard norms  \eqref{eq:unorm}. This is the choice $P=P_1$, where $P_1$ is orthogonal projection, so that $R(P_1) = \widetilde H^1(\Gamma^c)^\perp$ and, by standard Hilbert space theory (e.g., \cite[Theorem 4.11]{Rudin}),
\begin{equation} \label{eq:normbH}
\|\gamma u\|_{\bH(\Gamma)} = \|P_1 u\|_{H^1(\R^n)}, \qquad u\in H^1(\R^n),
\end{equation}
so that $\gamma:\widetilde H^1(\Gamma^c)^\perp \to \bH(\Gamma)$ and $\gamma^*:\bH^*(\Gamma)\to H^{-1}_\Gamma$ are unitary isomorphisms (cf.~\cite[Theorem 3.9]{caetano2025integral}). As stated in \S\ref{sec:main} we denote $A_k$ by $A_{1,k}$ in this case $P=P_1$.  It follows from \eqref{eq:bAkdef} that $\bA_k$ and $A_{1,k}$ are unitarily equivalent, in particular, as an operator $\bA_k:\bH^*(\Gamma)\to \bH(\Gamma)$,
\begin{equation} \label{eq:bAknorm}
\|\bA_{k}\| = \|A_{1,k}\|, \quad \mbox{and} \quad \|\bA_{k}^{-1}\| = \|A_{1,k}^{-1}\|
\end{equation}
if $A_{1,k}$ is invertible, i.e., if $k^2\not\in \sigma(-\Delta_D(\Omega_-))$. 

Our focus is, to an even greater extent, on the canonical choice for $P$ in the case that $H^{\pm1}(\R^n)$ are equipped with the $k$-dependent norms $\|\cdot\|_{H_k^{\pm 1}(\R^n)}$, given by \eqref{eq:knorm}. As discussed in \S\ref{sec:main} this is the choice $P=P_k$, where $P_k$ is orthogonal projection onto $R(P_k)= \widetilde H^1(\Gamma^c)^{\perp_k}$, the orthogonal complement of $\widetilde H^1(\Gamma^c)$ when $H^1(\R^n)$ is equipped with the norm \eqref{eq:knorm}. Let $\bH_k(\Gamma)$ denote $\bH(\Gamma)$ equipped with the norm $\|\cdot\|_{\bH_k(\Gamma)}$, equivalent to the norm $\|\cdot\|_{\bH(\Gamma)}$, defined by the right hand  side of \eqref{eq:bHnormdef} with the norm $\|u\|_{H^1(\R^n)}$ replaced by the norm $\|u\|_{H^1_k(\R^n)}$. Similarly, let $\bH_k^*(\Gamma)$ denote $\bH^*(\Gamma)$ equipped with the $k$-dependent dual space norm $\|\cdot\|_{\bH_k^*(\Gamma)}$ defined by
$$
\|\bpsi\|_{\bH^*_k(\Gamma)} := \sup_{0\neq \bphi \in \bH_k(\Gamma)}\frac{|\langle\!\langle \bpsi,\bphi\rangle\!\rangle|}{\|\bphi\|_{\bH_k(\Gamma)}}, \qquad \bpsi \in \bH^*(\Gamma),
$$
where $\langle\!\langle \cdot,\cdot \rangle\!\rangle$ denotes the duality pairing on $\bH^*(\Gamma)\times \bH(\Gamma)$ that is the extension of the inner product $(\cdot,\cdot)_{L^2(\Gamma,\cH^d)}$.
Then (cf.~\eqref{eq:normbH})
\begin{equation} \label{eq:normbHk}
\|\gamma u\|_{\bH_k(\Gamma)} = \|P_k u\|_{H_k^1(\R^n)}, \qquad u\in H^1(\R^n),
\end{equation}
so that $\gamma:\widetilde H^1(\Gamma^c)^{\perp_k} \to \bH_k(\Gamma)$ and $\gamma^*:\bH_k^*(\Gamma)\to H^{-1}_\Gamma$ are unitary isomorphisms (with  $H^{\pm 1}(\R^n)$ and their closed subspaces equipped with the norm \eqref{eq:knorm}). As noted in \S\ref{sec:main} we denote $A_k$ by $A_{k,k}$ in the case $P=P_k$, so that, by \eqref{eq:bAkdef},  $A_{k,k}$ and $\bA_k:\bH^*_k(\Gamma)\to \bH(\Gamma)$ are unitarily equivalent. Thus, where $\|\bA_k\|_k$ and $\|\bA_k^{-1}\|_k$ denote, respectively, the norms of $\bA_k:\bH^*_k(\Gamma)\to \bH_k(\Gamma)$ and its inverse, we have that
\begin{equation} \label{eq:bAknorm2}
\|\bA_{k}\|_k = \|A_{k,k}\|_k, \quad \mbox{and} \quad \|\bA_{k}^{-1}\|_k = \|A_{k,k}^{-1}\|_k
\end{equation}
if $k^2\not\in \sigma(-\Delta_D(\Omega_-))$. 

While we will not state these explicitly, bounds on $\bA_k$ and its inverse, as mappings between $\bH_k(\Gamma)$ and $\bH^*_k(\Gamma)$, can be deduced from \eqref{eq:bAknorm2} and the bounds on $A_{k,k}$ and its inverse in Propositions \ref{prop:Ak}, \ref{prop:Ak2}, and \ref{prop:lower}, Theorem \ref{thm:main}, and Corollaries \ref{cor:ss}, \ref{cor:Cinf}, and \ref{cor:gen}. Similarly, bounds on $\bA_k$ and its inverse, as mappings between $\bH(\Gamma)$ and $\bH^*(\Gamma)$, can be deduced from \eqref{eq:bAknorm} and bounds on $A_{1,k}$ and its inverse. These can, in turn, be deduced, through \eqref{eq:equiv2} and \eqref{eq:equiv2a}, from the bounds on $A_{k,k}$ and its inverse just referenced. 

As one example, from \eqref{eq:bAknorm}, \eqref{eq:equiv2}, and Proposition \ref{prop:Ak} it follows that, given any $k_0>0$ there exists $c>0$ such that $\|\bA_k\|\leq ck$, for $k\geq k_0$. It is worth noting that this leads to the  bound
\begin{equation} \label{eq:AL2bound}
\|\bA_k\|_{L^2} \leq Ck, \qquad k\geq k_0,
\end{equation}
for $\|\bA_k\|_{L^2}$, the norm of $\bA_k$ as a mapping on $L^2(\Gamma,\cH^d)$, by considering this mapping as the composition of the embedding $\iota^*:L^2(\Gamma,\cH^d)\to\bH^*(\Gamma)$, the mapping $\bA_k:\bH^*(\Gamma)\to \bH(\Gamma)$, and the embedding $\iota:\bH(\Gamma)\to L^2(\Gamma,\cH^d)$ (note that $\iota^*$ is the adjoint of $\iota$). Indeed, \eqref{eq:AL2bound} holds with $C= \|\iota\|^2 c$. 

In the case that $\Gamma$ is the boundary of a Lipschitz domain this estimate can be compared to known bounds on the $L^2$ norm of the standard single-layer BIO $S_k$, given by \eqref{eq:Sk}, since $\Gamma$ is a $d$-set with $d=n-1$ and $\cH^{n-1}$ coincides with surface measure (e.g., \cite[Theorem 3.8]{EvansGariepy}), so that $S_k=\bA_k$. We make this comparison in the comments before Corollary \ref{cor:L2}, which is proved in \S\ref{sec:slp2}, but first, in the following proposition, sharpen the bound \eqref{eq:AL2bound} by a more careful application of the above arguments, using that the mapping  $\gamma:H^s(\R^n)\to L^2(\Gamma,\cH^d)$ (and hence also $\gamma^*:L^2(\Gamma,\cH^d)\to H^{-s}(\R^n)$) is continuous not just for $s=1$ but for any $s>(n-d)/2$ \cite[Theorem 1]{Jonsson79}.

\begin{proposition} \label{prop:L2bound} Suppose that $\Gamma$ is a $d$-set with $n-2<d\leq n$. Then, given any $k_0>0$ and $\eps>0$, there exists $C>0$ such that
$$
\|\bA_k\|_{L^2} \leq Ck^{-1+\eps+n-d}, \qquad k\geq k_0.
$$
\end{proposition}
\begin{proof} Clearly it is enough to show the above bound in the case that $0<\eps<2 +d-n$.
Arguing as we did to obtain \eqref{eq:AL2bound}, we have that
$$
\|\bA_k\|_{L^2} \leq \|\iota^*\|_k \|\bA_k\|_k \|\iota\|_k,
$$ 
where $\|\iota^*\|_k$ and $\|\iota\|_k$ denote the norms of the embeddings $\iota^*:L^2(\Gamma,\cH^d)\to \bH_k^*(\Gamma)$ and $\iota:\bH_k(\Gamma)\to L^2(\Gamma, \cH^d)$, respectively, and note that $\|\iota\|_k=\|\iota^*\|_k$. Note further that $\gamma^*:\bH_k^*(\Gamma)\to H^{-1}_k(\R^n)$ is an isometry and $\gamma^*\iota^* \bpsi =\gamma^*\bpsi$, for $\bpsi\in L^2(\Gamma,\cH^d)$, with $\gamma^*\bpsi\in H^{-1}(\R^n)$ given by \eqref{eq:gamadj}. Thus, and since $\gamma^*:L^2(\Gamma,\cH^d)\to H^{-(\eps+n-d)/2}(\R^n)$ is continuous for $\eps>0$, and recalling \eqref{eq:equiv}, given $k_0>0$ and $0<\eps<2+d-n$ there exists $C',C''>0$ such that, for all $\bpsi\in L^2(\Gamma, \cH^d)$ and $k\geq k_0$, where $s:= (\eps+n-d)/2 \in (0,1)$,
\begin{eqnarray*}
\|\iota^*\bpsi\|_{\bH_k^*(\Gamma)} &= &\|\gamma^*\bpsi\|_{H^{-1}_k(\R^n)}\\
&\leq &k^{-1+s}\|\gamma^*\bpsi\|_{H_k^{-s}(\R^n)}\\
& \leq &C' k^{-1+s}\|\gamma^*\bpsi\|_{H^{-s}(\R^n)}\\
& \leq &C'' k^{-1+s}\|\bpsi\|_{L^2(\Gamma,\cH^d)}.
\end{eqnarray*}
Thus $\|\iota\|_k = \|\iota^*\|_k \leq C'' k^{-1+s}$, for $k\geq k_0$, and the result follows by \eqref{eq:bAknorm2} and Proposition \ref{prop:Ak}.
\end{proof}

The norm $\|\bA_k\|_{L^2}$ can also be estimated via the representation \eqref{eq:bAk2}. The following is one such estimate, which is sharper than Proposition \ref{prop:L2bound} in its dependence on $k$ only when $n-2<d\leq (n+1)/2$, which is possible only for $n\leq 4$.

\begin{proposition} \label{prop:L2bound2} Suppose that $\Gamma$ is a $d$-set with $n-2<d\leq n$. Then, given any $k_0>0$, there exists $C_1,C_2>0$ such that
$$
\|\bA_k\|_{L^2} \leq C_1 k^{(n-3)/2}, \qquad k\geq k_0,
$$
if $n\geq 3$, and
$$
\|\bA_k\|_{L^2} \leq C_2 \left\{\begin{array}{ll} k^{-d}, & \mbox{if } d<1/2,\\  k^{-1/2}\log(1+k), & \mbox{if } d=1/2,\\ k^{-1/2}, & \mbox{if }d> 1/2,\end{array}\right.
$$
if $n=2$.
\end{proposition}
\begin{proof} We adapt the argument of \cite[Theorem 3.3]{ChGrLaLi:09} which proves this result in the case that $n=2$ or 3 and $\Gamma$ is the boundary of a Lipschitz domain, so $d=n-1$. Arguing as in the proof of that theorem we have, by Riesz-Thorin interpolation (e.g., \cite[Theorem 1.1.1]{Bergh:76}) and since $\Phi_k(x,y)=\Phi_k(y,x)$, that $\|\bA_k\|_{L^2} \leq \max(\|\bA_k\|_{L^\infty},\|\bA_k\|_{L^1}) = \|\bA_k\|_{L^\infty}$, where $\|\bA_k\|_{L^p}$ denotes the norm of $\bA_k$ as an operator on $L^p(\Gamma,\cH^d)$, for $1\leq p\leq \infty$. Further, from \eqref{eq:bAk2}  (see, e.g., \cite{Jorgens}),
$$
\|\bA_k\|_{L^\infty} = \mathrm{ess}\, \sup_{x\in \Gamma} \int_\Gamma|\Phi_k(x,y)| \, \rd \cH^d(y) = k^{n-2} \mathrm{ess}\,\sup_{x\in \Gamma} \int_\Gamma f(k|x-y|) \, \rd \cH^d(y),
$$
by \eqref{eq:Phidef}, where $f(r) := (2\pi r)^{-(n-2)/2}|H^{(1)}_{(n-2)/2}(r)|/4$, $r>0$. Since $|H^{(1)}_\nu|$ is decreasing on $(0,\infty)$ for $\nu\geq 0$ \cite[\S13.74]{Watson}, it follows from \cite[Lemma 2.1]{caetano2024} that, for some constant $C>0$ independent of $k$, and where $D:= \mathrm{diam}(\Gamma)$, 
\begin{equation} \label{eq:bAb2}
\|\bA_k\|_{L^\infty} \leq C k^{n-2} \int_0^D r^{d-1}f(kr)\, \rd r = C k^{n-2-d} \int_0^{kD} t^{d-1} f(t)\, \rd t.
\end{equation} 
By standard asymptotics of Bessel functions \cite[\S10.8]{NIST},  $f(t)=O(\log(t))$ as $t\to 0^+$ if $n=2$, $= O(t^{2-n})$ as $t\to 0^+$ if $n>2$. Thus, and since $d>n-2$, the integrals in \eqref{eq:bAb2}  are well-defined. Further, by \cite[\S10.17(i)]{NIST},  $f(t) = O(t^{-(n-1)/2})$ as $t\to \infty$ which, together with \eqref{eq:bAb2} and recalling that $\|\bA_k\|_{L^2}\leq \|\bA_k\|_{L^\infty}$, implies the required result.
\end{proof}

\section{The classical single-layer BIO and its inverse} \label{sec:slp2}
Where $S_k$ is the classical single-layer BIO, defined by \eqref{eq:Sk},  in this section we establish the bounds on $S_k$ and $S_k^{-1}$ stated as Corollaries \ref{thm:Sk0}--\ref{cor:L2}.

Suppose that $\Omega^{*}$ is a bounded Lipschitz domain. In the case that $O=\overline{\Omega^{*}}$ and $\Gamma = \partial O$, $\Gamma$ is a $d$-set with $d=n-1$ and the Hausdorff measure $\cH^d$ coincides with surface measure on $\Gamma$, as noted before Proposition \ref{prop:L2bound}, so that $L^2(\Gamma,\cH^d)=L^2(\Gamma)$. Further, as discussed in \cite[Remark 3.10]{caetano2025integral} and \cite[Remark 3.6]{caetano2021}, $H^{1/2}(\Gamma)=\bH(\Gamma)$, with equivalence of norms, so that the Gelfand triples $(H^{1/2}(\Gamma),L^2(\Gamma),H^{-1/2}(\Gamma))$ and $(\bH(\Gamma),L^2(\Gamma,\cH^{n-1}),\bH^*(\Gamma))$ coincide, in particular $\bH^*(\Gamma)=H^{-1/2}(\Gamma)$ with equivalence of norms. Further, by \eqref{eq:bAk2} and \eqref{eq:Sk}, $\bA_k\bpsi=S_k\bpsi$, for $\psi\in L^\infty(\Gamma)=L^\infty(\Gamma,\bH^{n-1})$, which is dense in $H^{-1/2}(\Gamma)=\bH^*(\Gamma)$. Thus Corollary \ref{thm:Sk0}, and Corollary \ref{cor:L2} as it applies to the case that $\Gamma$ is the boundary of a Lipschitz domain, follow from \eqref{eq:bAknorm}, \eqref{eq:equiv2}, \eqref{eq:equiv2a}, and Propositions \ref{prop:Ak}, \ref{prop:L2bound}, and \ref{prop:L2bound2}. 

To see the validity of Corollary \ref{cor:L2} as it applies to the Lipschitz screen case, suppose that $U$ is a Lipschitz screen in the sense of \S\ref{sec:boundary}, $\Gamma:= \overline{U}$, and $\Pi := \Gamma \setminus U$. Then, for some bounded Lipschitz domain $\Omega^{*}$ and where $\Gamma':= \partial \Omega^{*}$, $U$ and $V:= \Gamma'\setminus \Gamma$ are relatively open subsets of $\Gamma'$ and $(U,\Pi,V)$ is a Lipschitz dissection of $\Gamma'$ in the sense of \S\ref{sec:boundary}. It follows that $\Gamma$ is a $d$-set with $d=n-1$ and, as noted above, the Hausdorff measure $\cH^d$ coincides with surface measure on $\Gamma$. Further, $\Pi$ has zero surface measure, i.e., $\cH^{n-1}(\Pi)=0$, so that $L^2(\Gamma, \cH^{n-1})=L^2(\Gamma)=L^2(U)$. Viewing $L^2(\Gamma)$ as a closed subspace of $L^2(\Gamma')$, Corollary \ref{cor:L2} in the Lipschitz screen case follows from  the same corollary applied in the case of the Lipschitz boundary $\Gamma'$.

To see that Corollary \ref{thm:Sk} also holds, continuing to use the notation of the previous paragraph, let $\gamma:H^1(\R^n)\to L^2(\Gamma,\cH^{n-1})=L^2(U)$ and $\gamma':H^1(\R^n)\to L^2(\Gamma',\cH^{n-1})=L^2(\Gamma')$ denote the continuous trace operators introduced in \S\ref{sec:dset}, and let $|_{U}:L^2(\Gamma')\to L^2(U)$ denotes the restriction operator.  Clearly $\gamma u = |_U\circ \gamma'u$, for $u\in \cS(\R^n)$, and hence for $u\in H^1(\R^n)$. Since also $|_{U}:H^{1/2}(\Gamma')\to H^{1/2}(U)$ is continuous and surjective and, as noted above, $\bH(\Gamma')=H^{1/2}(\Gamma')$, we see that, as sets, $\bH(\Gamma)=H^{1/2}(U)$. To see that the norms on $\bH(\Gamma)$ and $H^{1/2}(U)$ are equivalent, recall that we have noted above that the norms on $\bH(\Gamma')$ and $H^{1/2}(\Gamma')$ are equivalent; in particular, for some $c>0$, $\|\bphi\|_{\bH(\Gamma')} \leq c\|\bphi\|_{H^{1/2}(\Gamma')}$, $\bphi\in H^{1/2}(\Gamma')$. Further,
suppose that $\bpsi \in \bH(\Gamma)=H^{1/2}(U)$. Then there exists $u\in H^1(\R^n)$ such that $\gamma u = \bpsi$ and
$\|u\|_{H^1(\R^n)} \leq 2\|\bpsi\|_{\bH(\Gamma)}$. Thus, and where $c'$ denotes the norm of $\gamma':H^1(\R^n)\to H^{1/2}(\Gamma')$,  
$$
\|\bpsi\|_{H^{1/2}(U)} = \| |_U\circ \gamma' u\|_{H^{1/2}(U)} \leq \|\gamma'u\|_{H^{1/2}(\Gamma')} \leq c' \|u\|_{H^1(\R^n)} \leq 2c'\|\bpsi\|_{\bH(\Gamma)}.
$$ 
Similarly, there exists $\bphi\in H^{1/2}(\Gamma')$ and $v\in H^1(\R^n)$ such that $\bpsi = \bphi|_U$, $\gamma' v=\bphi$, $\|\bphi\|_{H^{1/2}(\Gamma')} \leq 2 \|\bpsi\|_{H^{1/2}(U)}$, and 
$$
\|v\|_{H^1(\R^n)}\leq 2\|\bphi\|_{\bH(\Gamma')} \leq 2c\|\bphi\|_{H^{1/2}(\Gamma')},
$$
so that, and since $\gamma v = |_U\circ \gamma' v = \bpsi$,
$$
\|\bpsi\|_{\bH(\Gamma)} \leq \|v\|_{H^1(\R^n)}  \leq 2c\|\bphi\|_{H^{1/2}(\Gamma')} \leq 4c\|\bpsi\|_{H^{1/2}(U)}.
$$

We have shown that $L^2(\Gamma, \cH^{n-1})=L^2(U)$, and that $\bH(\Gamma)=H^{1/2}(U)$, with equivalence of norms. Thus the two Gelfand triples, $(H^{1/2}(U),L^2(U),$ $\widetilde H^{-1/2}(U))$ and $(\bH(\Gamma),L^2(\Gamma,\cH^{n-1}),\bH^*(\Gamma))$ coincide, in particular $\bH^*(\Gamma)=\widetilde H^{-1/2}(U)$, with equivalence of norms. Further, by \eqref{eq:bAk2} and \eqref{eq:Sk}, $\bA_k\bpsi=S_k\bpsi$, for $\psi\in L^\infty(U)=L^\infty(\Gamma,\bH^{n-1})$, which is dense in $\bH^*(\Gamma)=\widetilde H^{-1/2}(U)$. Thus Corollary \ref{thm:Sk} follows from \eqref{eq:bAknorm}, \eqref{eq:equiv2}, \eqref{eq:equiv2a}, Proposition \ref{prop:Ak}, and Corollaries \ref{cor:ss} and \ref{cor:gen}.

\subsection*{Acknowledgements}
\small The authors are grateful to David Lafontaine, Euan Spence, Melissa Tacy, and Jared Wunsch for stimulating discussions in relation to this work. The first and second authors acknowledge, respectively, the support of the New Zealand Marsden Fund (Grant No.~MFP-UOA2527), and the support of a UK Engineering and Physical Sciences Research Council (EPSRC) PhD Studentship. \normalsize

\subsection*{Availability of data and materials} \small The authors confirm that the research did
not involve the generation or analysis of any datasets.


\end{document}